\documentclass[a4paper,11pt,final]{article}

\usepackage{amsmath}
\usepackage{amsfonts,amsthm}
\usepackage{graphicx,color}
\usepackage{hyperref}
\usepackage{enumitem}
\usepackage{geometry}
\usepackage[utf8x]{inputenc}

\newtheorem{proposition}{Proposition}[section]
\newtheorem{theorem}[proposition]{Theorem}
\newtheorem{lemma}[proposition]{Lemma}
\newtheorem{corollary}[proposition]{Corollary}
\newtheorem{definition}[proposition]{Definition}
\newtheorem{remark}[proposition]{Remark}
\newenvironment{proofof}[1]{\smallskip\noindent{\textbf{Proof~of~#1.}}%
  \hspace{1pt}}{\hspace{-5pt}{\nobreak\quad\nobreak\hfill\nobreak%
    $\square$\vspace{2pt}\par}\smallskip\goodbreak}

\numberwithin{equation}{section}
\allowdisplaybreaks[4]

\renewcommand{\phi}{\varphi}
\renewcommand{\theta}{\vartheta}
\renewcommand{\epsilon}{\varepsilon}
\renewcommand{\L}[1]{\mathbf{L^#1}}

\newcommand{\C}[1]{\mathbf{C^{#1}}}

\newcommand{\Cc}[1]{\mathbf{C_c^{#1}}}

\newcommand{\BV}{\mathbf{BV}}
\newcommand{\BVloc}{\mathbf{BV_{loc}}}
\newcommand{\modulo}[1]{{\left|#1\right|}}
\newcommand{\norma}[1]{{\left\|#1\right\|}}
\newcommand{\reali}{{\mathbb{R}}}

\newcommand{\Lip}{\mathop\mathbf{Lip}}
\newcommand{\tv}{\mathop\mathrm{TV}}

\newcommand{\pint}[1]{\mathaccent23{#1}}

\renewcommand{\d}[1]{\mathinner{\mathrm{d}{#1}}}

\newcommand{\XX}{\mathcal{X}}
\newcommand{\dx}{d_{\mathcal{X}^n}}
\newcommand{\supess}{\mathop{\rm ess~sup}}

\newcommand{\II}{\mathbb{I}}
\newcommand{\tmax}{{t_*}} 

\renewcommand{\marginpar}[1]{{}}
\begin{document}

\title{Well Posedness and Control in a NonLocal SIR Model}

\author{Rinaldo M.~Colombo$^1$ \and Mauro Garavello$^2$}

\footnotetext[1]{INdAM Unit, University of
  Brescia. \texttt{rinaldo.colombo@unibs.it}}

\footnotetext[2]{Department of Mathematics and its Applications,
  University of Milano - Bicocca. \par
  \texttt{mauro.garavello@unimib.it}}

\maketitle

\begin{abstract}

  \noindent $SIR$ models, also with age structure, can be used to
  describe the evolution of an infective disease. A vaccination
  campaign influences this dynamics immunizing part of the susceptible
  individuals, essentially turning them into recovered individuals. We
  assume that vaccinations are dosed at prescribed times or ages which
  introduce discontinuities in the evolutions of the $S$ and $R$
  populations. It is then natural to seek the \emph{``best''}
  vaccination strategies in terms of costs and/or effectiveness. This
  paper provides the basic well posedness and stability results on the
  $SIR$ model with vaccination campaigns, thus ensuring the existence
  of optimal dosing strategies.

  \medskip


  \medskip

\end{abstract}


\section{Introduction}
\label{sec:I}

Aim of this paper is to provide a rigorous analytic environment where
different vaccination strategies can be described, tested and
optimized.

Our starting point is the following age--structured \emph{Susceptible
  -- Infected -- Recovered} ($SIR$) model, which originated
in~\cite{KermackMcKendrick1927}, see also~\cite[Chapter~6]{Inaba2017},
\cite[Chapter~19]{MurrayBook} or~\cite[\S~1.5.1]{PerthameBook},
\begin{equation}
  \label{eq:13}
  \begin{array}{c@{\;}c@{\;}c@{\;}c@{\;}c@{\;}c@{\;}c@{\;}c@{\;}c@{\;}c}
    \partial_t S
    & +
    & \partial_a S
    & =
    & -
    &  d_S(t, a) \, S
    & -
    &  \int_0^{+\infty} \lambda(a, a') \; I(t, a') \d a' \, S
    \\
    \partial_t I
    & +
    & \partial_a I
    & =
    & -
    & d_I(t, a) \, I
    & +
    & \int_0^{+\infty} \lambda(a, a') \, I(t, a') \d a' \, S
    & -
    & r_I(t, a) \, I
    \\
    \partial_t R
    & +
    & \partial_a R
    & =
    & -
    & \underbrace{d_R(t, a) \, R}
    &
    &\underbrace{\vphantom{d_R(}
      \hphantom{\int_0^{+\infty} \lambda(a, a') \, I(t, a') \d a' \, S}}
    & +
    & \underbrace{r_I(t, a) \, I} .
    \\
    &&&&& \mbox{mortality}
    &
    & \mbox{disease transmission}
    &
    & \mbox{recovery}
  \end{array}
\end{equation}
As usual, $S = S(t, a)$ is the density of individuals at time $t$ of
age $a$ susceptible to the disease; $I = I(t, a)$ is the density of
infected individuals at time $t$ and of age $a$ and the density of
individuals that can not be infected by the disease is $R = R(t, a)$,
comprising individuals that recovered from the disease as well as
those that are immune. The death rates of the three portions of the
populations are $d_S$, $d_I$ and $d_R$. Above, $\lambda (a, a')$
quantifies the susceptible individuals of age $a$ that are infected by
individuals of age $a'$. Thus, the nonlocal term
$\int_0^{+\infty} \lambda(a, a') \; I(t, a') \d a' \, S (t,a)$ in the
former two right hand sides represents the total number of susceptible
individuals of age $a$ that become infected at time $t$. Finally,
$r_I (t,a)$ is the fraction of infected individuals of age $a$ that
recover at time $t$, independently from the vaccination campaign.

We now introduce a vaccination campaign in~\eqref{eq:13}. To this aim,
differently from various paper in the literature,
e.g.~\cite{ElAlamiEtAl, HadelerMueller2007, Mueller1999,Mueller2000,
  WangEtAl2019, ZamanEtAl2017}, we do not introduce any source term in
the right hand sides of~\eqref{eq:13}. We consider two different
approaches.

In a first policy, vaccinations are dosed at a, possibly time
dependent, percentage of the population of the prescribed ages
$\bar a_1$, $\bar a_2$, $\ldots$, $\bar a_N$, with
$\bar a_{j-1} < \bar a_j$. Call $\eta_j (t)$, with
$\eta_j (t) \in [0,1]$ the fraction of the $S$ population of age
$\bar a_j$ that is dosed a vaccine at time $t$. Then, assuming that
vaccination has an immediate effect, the evolution described
by~\eqref{eq:13} has to be supplemented by the vaccination conditions
\begin{equation}
  \label{eq:1}
  \begin{array}{@{}rclr@{}}
    S(t, \bar a_j+)
    & =
    & \left(1 - \eta_j(t)\right) S(t, \bar a_j-)
    & \mbox{[$\forall t$, $S (t, \bar a_j)$ decreases due to vaccination]}
    \\
    I(t, \bar a_j+)
    & =
    & I(t, \bar a_j-)
    & \mbox{[the infected population is unaltered]}
    \\
    R(t, \bar a_j+)
    & =
    & R(t, \bar a_j-) + \eta_j(t) S(t, \bar a_j-)
    & \mbox{[vaccinated individuals are immunized]}
  \end{array}
\end{equation}
for a.e.~$t > 0$ and for every $j \in \{1, \cdots, N\}$. Whenever
vaccinations can be dosed only to susceptible individuals, the total
cost of the vaccination campaign~\eqref{eq:1} at all ages
$\bar a_1, \ldots, \bar a_N$ is proportional to the total number of
vaccinations dosed, say
\begin{equation}
  \label{eq:constraint}
  \mathcal{C} (\eta)
  =
  \sum_{i = 1}^N \int_{\II} \eta_i(t)  \, S (t, \bar a_i-) \d t \,,
\end{equation}
$\II$ being the time interval under consideration and $S$ depending on
$\eta$ through~\eqref{eq:1}. However, it is reasonable to consider
also the case of vaccinations dosed to the $\eta_j (t)$ portion of the
\emph{whole} population at time $t$, that is also to infected and
immune individuals, in which case~\eqref{eq:constraint} has to be
substituted by
\begin{equation}
  \label{eq:35}
  \mathcal{C} (\eta)
  =
  \sum_{i = 1}^N \int_{\II} \eta_i(t)
  \left(
    S (t, \bar a_i-) + I (t, \bar a_i-) + R (t, \bar a_i-)
  \right) \d t \,,
\end{equation}
where $S$, $I$ and $R$ depend on $\eta$ through~\eqref{eq:1}. Indeed,
not always individuals belonging to the $R$ or even $I$ population can
be easily distinguished from those in the $S$ population. Remark that
in both cases~\eqref{eq:constraint} and~\eqref{eq:35}, the dynamics of
the disease is described by~\eqref{eq:13}--\eqref{eq:1}, since
vaccination is assumed to have no effects on $R$ or $I$ individuals.

Alternatively, in a second policy, a vaccination campaign may aim at
immunizing an age dependent portion, say
$\nu_1 (a), \ldots, \nu_N (a)$, of the $S$ population at given times
$\bar t_1, \ldots, \bar t_N$. This amounts to substitute~\eqref{eq:1}
with
\begin{equation}
  \label{eq:14}
  \begin{array}{@{}rclr@{}}
    S (\bar t_k+, a)
    & =
    & \left(1-\nu_k (a)\right) \, S (\bar t_k-, a)
    & \mbox{[$\forall a$, $S (\bar t_k, a)$ decreases due to vaccination]}
    \\
    I (\bar t_k+, a)
    & =
    & I (\bar t_k-, a)
    & \mbox{[the infected population is unaltered]}
    \\
    R (\bar t_k+, a)
    & =
    & R (\bar t_k-, a) + \nu_k (a) \, S (\bar t_k-, a)
    & \mbox{[vaccinated individuals are immunized].}
  \end{array}
\end{equation}
Now, a reasonable cost due to this campaign is
\begin{equation}
  \label{eq:8}
  \mathcal{C} (\nu)
  =
  \sum_{k = 1}^N \int_{\reali^+} \nu_k(a)  \, S (\bar t_k-, a) \d a
\end{equation}
whenever vaccination is dosed only to susceptible individuals. On the
other hand, vaccinations can be dosed to \emph{all} individuals, in
which case we replace~\eqref{eq:8} with
\begin{equation}
  \label{eq:39}
  \mathcal{C} (\nu)
  =
  \sum_{k = 1}^N \int_{\reali^+} \nu_k(a)
  \left(S (\bar t_k-, a) + I (\bar t_k-, a) + R (\bar t_k-, a)\right)
  \d a \,.
\end{equation}
As above, in both cases~\eqref{eq:8} and~\eqref{eq:39}, the dynamics
of the disease is described by~\eqref{eq:13}--\eqref{eq:14}, since
vaccination is assumed to have no effects on $R$ or $I$ individuals.

The most natural way to evaluate the \emph{effect} of a vaccination
campaign is to compute the, possibly weighted, number of infected
individuals, namely
\begin{equation}
  \label{eq:7}
  \mathcal{E}
  =
  \int_{\II} \int_{\reali^+} \phi(t,a) \, I(t, a) \d a\, \d t \,,
\end{equation}
$\mathcal{E}$ being a function of $\eta$ in case~\eqref{eq:1} and a
function of $\nu$ in case~\eqref{eq:14}.  The dependence of the weight
$\phi$ on time $t$ may account for a possible targeting a decrease in
the total number of infected individuals after an initial period,
while the dependence of $\phi$ on $a$ may account for different
degrees of danger of the disease at the different ages.

Once the cost $\mathcal{C}$ and the effect $\mathcal{E}$ are selected,
we are left with two modeling choices: \emph{``The optimization
  problem in this framework is to find the strategy with minimal costs
  at a given level for the effect or to find the strategy with the
  best effect at given costs.''},
from~\cite[Introduction]{Mueller2000}. In more formal terms, we are
lead to tackle the problems
\begin{equation}
  \label{eq:40}
  \mbox{minimize } \mathcal{C}
  \mbox{ subject to } \mathcal{E} \leq \mathcal{E}_*
  \qquad \mbox{ or } \qquad
  \mbox{minimize } \mathcal{E}
  \mbox{ subject to } \mathcal{C} \leq \mathcal{C}_*
\end{equation}
for assigned positive $\mathcal{E}_*$ and $\mathcal{C}_*$, with time
dependent controls $\eta_i$ in cases~\eqref{eq:constraint}
or~\eqref{eq:35}, or else with age dependent controls $\nu_k$ in
cases~\eqref{eq:8} or~\eqref{eq:39}. The analytic results presented below provide a framework, consisting of well posedness results and stability estimates, where these problems can be rigorously addressed, see~\cite{Leiden2018} for soem specific examples.

\smallskip

The current literature offers a variety of alternative approaches to
similar modeling situations. For instance, in the
recent~\cite{ElAlamiEtAl}, the vaccination control enters an equation
for $S$ similar to that in~\eqref{eq:13} through a term $-u S$ in the
right hand side, meaning that vaccination takes place uniformly at all
ages. A similar approach is followed also in~\cite{Mueller1999,
  Mueller2000}.

\smallskip

From the analytic point of view, below we prove well posedness and
stability for~\eqref{eq:13}--\eqref{eq:1} and
for~\eqref{eq:13}--\eqref{eq:14} which, in turn, ensure the existence
of optimal vaccination strategies. To achieve this, we prove well
posedness and stability of a more general IBVP, see~\eqref{eq:2}.

\medskip

The next section presents solutions to problems~\eqref{eq:40}, as a
consequence of the analytic framework developed in
Section~\ref{sec:A}. All analytic proofs are collected in
Section~\ref{sec:P}.

\section{The Controlled SIR Models}
\label{sec:CSIR}

Denote by $\II$ the time interval $[0, T]$, for a positive $T$, or
$\left[0, +\infty\right[$.

Throughout, we supplement~\eqref{eq:13} with the initial and boundary
conditions
\begin{equation}
  \label{eq:IC+BDY}
  \left\{
    \begin{array}{r@{\,}c@{\,}l@{\quad}%
      r@{\,}c@{\,}l@{\quad}r@{\,}c@{\,}l@{\qquad}r@{\,}c@{\,}l}
      S(t, 0)
      & =
      & S_b(t),
      & I(t, 0)
      & =
      & I_b (t),
      & R(t, 0)
      & =
      & R_b (t),
      & t
      & \in
      & \II \,,
      \\
      S(0, a)
      & =
      & S_o(a),
      & I(0, a)
      & =
      & I_o(a),
      & R(0, a)
      & =
      & R_o(a),
      & a
      & \in
      & \reali^+ ;
    \end{array}
  \right.
\end{equation}
We require below the following assumptions on the functions
defining~\eqref{eq:13}--\eqref{eq:IC+BDY}\footnote{Throughout, we
  strive to have dimensionally correct expressions at the cost of
  distinguishing the various constants whenever they are dimensionally
  different.}:
\begin{enumerate}[label=\textbf{(\boldmath{$\lambda$})}, align=left]
\item \label{hyp:lambda}
  $\lambda \in \C0(\reali^+ \times \reali^+; \reali)$ admits positive
  constants $\Lambda_\infty, \Lambda_L$ such that for all
  $a_1, a_2, a' \in \reali^+$
  \begin{eqnarray}
    \label{eq:Lambda2}
    \norma{\lambda}_{\L\infty (\reali^+\times \reali^+; \reali)}
    +
    \tv (\lambda(\cdot, a'); \reali^+)
    & \le
    & \Lambda_\infty
    \\
    \label{eq:Lambda}
    \modulo{\lambda(a_1, a') - \lambda(a_2, a')}
    & \le
    & \Lambda_L \; \modulo{a_1 - a_2}
      \,.
  \end{eqnarray}
\end{enumerate}
\begin{enumerate}[label=\textbf{(dr)}, align=left]
\item \label{hyp:dr} The maps
  $d_S, d_I, d_R, r_I \colon \II \times \reali^+ \to \reali$ are
  Caratheodory functions, in the sense of Definition~\ref{def:Cara},
  and there exist positive $R_L,R_1, R_\infty$ such that for
  $\phi = d_S, d_I, d_R, r_I$, $t \in \II$ and $a_1,a_2 \in \reali^+$,
  \begin{eqnarray}
    \label{eq:24}
    \norma{\phi}_{\L\infty (\II \times \reali^+; \reali)}
    +
    \tv (\phi (t, \cdot); \reali^+)
    & \leq
    & R_\infty
    \\
    \label{eq:22}
    \modulo{\phi (t,a_2) - \phi (t, a_1)}
    & \leq
    & R_L \, \modulo{a_2 - a_1}
    \\
    \label{eq:27}
    \norma{\phi}_{\C0(\II ; \L1( \reali^+; \reali))}
    & \leq
    & R_1 \,.
  \end{eqnarray}
\end{enumerate}
\begin{enumerate}[label=\textbf{(IB)}, align=left]
\item \label{hyp:data} The initial and boundary data satisfy
  \begin{equation}
    \label{eq:iBV-data}
    S_o, I_o, R_o \in (\L1 \cap \BV) (\reali^+; \reali^+)
    \qquad \textrm{ and } \qquad
    S_b, I_b, R_b \in (\L1 \cap \BV) (\II; \reali^+) \,.
  \end{equation}
\end{enumerate}

\noindent First, we provide the basic well posedness result for the model
presented above, based on the nonlocal renewal equations~\eqref{eq:13},
in the case of the vaccination policy~\eqref{eq:1}.

\begin{theorem}
  \label{thm:wp}
  Under hypotheses~\ref{hyp:lambda} and~\ref{hyp:dr}, for any initial
  and boundary data satisfying~\ref{hyp:data}, for any choice of the
  vaccination ages $\bar a_1, \ldots, \bar a_N$ and of the control
  function $\eta \in \BV(\II; [0, 1]^N)$,
  problem~\eqref{eq:13}--\eqref{eq:1}--\eqref{eq:IC+BDY} admits a
  unique solution
  \begin{equation}
    \label{eq:42}
    (S,I,R) \in \C{0,1} \left(\II; \L1 (\reali^+; \reali^3)\right)
    \qquad \mbox{ with } \qquad
    \begin{array}{r@{\,}c@{\,}l}
      S (t,a)
      & \geq
      & 0
      \\
      I(t,a)
      & \geq
      & 0
      \\
      R (t,a)
      & \geq
      & 0
    \end{array}
    \mbox{ for all } (t,a) \in \II \times \reali^+,
  \end{equation}
  depending Lipschitz continuously on the initial datum, through its
  $\L1$ norm, and on $\eta$, through its $\L\infty$ norm.
\end{theorem}

The proof, deferred to \S~\ref{sec:ProofsCSIR}, amounts to show that
problem~\eqref{eq:13}--\eqref{eq:1}--\eqref{eq:IC+BDY} satisfies the
assumptions of Theorem~\ref{thm:general} and of
Corollary~\ref{cor:posGlog} below.

In the case of the vaccination policy~\eqref{eq:14}, we obtain an
analogous result.

\begin{theorem}
  \label{thm:wp2}
  Under hypotheses~\ref{hyp:lambda} and~\ref{hyp:dr}, for any initial
  and boundary data satisfying~\ref{hyp:data}, for any choice of the
  vaccination times $\bar t_1, \ldots, \bar t_N$ and of the control
  function $\nu \in \BV(\II; [0, 1]^N)$,
  problem~\eqref{eq:13}--\eqref{eq:14}--\eqref{eq:IC+BDY} admits a
  unique solution as in~\eqref{eq:42}, depending Lipschitz
  continuously on the initial datum, through the $\L1$ norm, and on
  $\nu$, through the $\L\infty$ norm.
\end{theorem}

\noindent The proof is deferred to \S~\ref{sec:ProofsCSIR}.

Once Theorem~\ref{thm:wp} and Theorem~\ref{thm:wp2} are acquired, both
costs $\mathcal{C}$ and $\mathcal{E}$ are easily shown to be strongly
continuous functions of the control $\eta$. The existence of an
optimal strategy then follows through an application of {Weierstra\ss}
Theorem, as soon as the choice of $\eta$ or $\nu$ is restricted to a
suitable strongly compact set. We refer to~\cite{Leiden2018} for a
selection of control problems based on Theorem~\ref{thm:wp} or Theorem~\ref{thm:wp2}.

\section{Analytic Results}
\label{sec:A}

\marginpar{\small$[u] = \frac{n}{m}$\\ $[g] = \frac{m}{s}$\\
  $[\alpha] = \frac{1}{s}$\\ $[\gamma]=\frac{1}{s}$\\
  $[\beta]=\frac{n}{s}$} The proofs of Theorem~\ref{thm:wp} and of
Theorem~\ref{thm:wp2} follow from a slightly more general statement.

\begin{theorem}
  \label{thm:general}
  Consider the following Initial -- Boundary Value Problem (IBVP)
  \begin{equation}
    \label{eq:2}
    \left\{
      \begin{array}{l}
        \partial_t u_i + \partial_x (g_i (t,x) \, u_i)
        =
        \left(\alpha_i [u (t)](x) + \gamma_i (t,x)\right) \cdot u
        \\
        u_i (0,x) = u_{o,i} (x)
        \\
        g_i (t,0+) \,  u_i (t,0+) = \beta_i \left(t, u_1 (t, \bar x_1-),
        \ldots, u_n (t, \bar x_n-)\right)
      \end{array}
    \right.
    \qquad i=1, \ldots, n
  \end{equation}
  where\marginpar{$[\check g] = \frac{m}{s}$\\$[G_\infty] =
    \frac{m}{s}$\\ $[G_1] = \frac1s$}
  \begin{enumerate}[label=\textbf{\textup{(IBVP.\arabic*)}},
    align=left]
  \item \label{hyp:g} $g_1, \cdots, g_n \in \C{0,1} (\II \times
    \reali^+; [\check g, \hat g])$ and for all $t \in \II$, $x \in
    \reali^+$, $i=1, \ldots, n$ \marginpar{$[\check g] =
      \frac{m}{s}$\\$[G_\infty] = \frac{m}{s}$\\ $[G_1] = \frac1s$}
    \begin{eqnarray}
      \label{eq:36}
      \tv(g_i (\cdot, x); \II)
      +
      \tv(g_i (t,\cdot); \reali^+)
      & \leq
      & G_\infty
      \\
      \label{eq:37}
      \tv(\partial_x g_i (t,\cdot); \reali^+)
      +
      \norma{\partial_x g_i (t, \cdot)}_{\L\infty(\reali^+; \reali)}
      & \leq
      & G_1 \,.
    \end{eqnarray}
  \item \label{hyp:A_i} \marginpar{{\small$[A_L] =\frac{1}{ns}$\\
        $[A_1] =\frac{1}{s}$\\ $[A_2] =\frac{1}{ms}$\\}}
    $\alpha_1, \ldots, \alpha_n \colon \L1 (\reali^+; \reali^n) \to
    \C0(\reali^+; \reali^n)$ are linear and continuous maps and there
    exist positive constants $A_L$ and $A_1$ such that
    \begin{align}
      \label{eq:hyp-A1}
      \norma{\alpha_i[w]}_{\L\infty(\reali^+; \reali^n)}
      +
      \tv (\alpha_i[w]; \reali^+)
      &\le
        A_L \; \norma{w}_{\L1(\reali^+; \reali^n)}
      \\
      \label{eq:hyp-A2}
      \norma{\alpha_i[w]}_{\L1(\reali^+; \reali^n)}
      & \le
        A_1 \; \norma{w}_{\L1(\reali^+; \reali^n)}
    \end{align}
    for every $i \in \left\{1, \cdots, n\right\}$,
    $w \in \L1(\reali^+; \reali^n)$. Moreover, for every
    $w \in (\L1 \cap \L\infty) (\reali^+; \reali)$, there exists a
    positive $A_2$ such that for all $x_1, x_2 \in \reali^+$
    \begin{equation}
      \label{eq:hyp-A4}
      \norma{\alpha_i[w](x_1) - \alpha_i[w](x_2)}
      \le
      A_2 \; \modulo{x_1 - x_2} \,.
    \end{equation}

  \item \label{hyp:gamma_i} \marginpar{\small$[C_L] = \frac{m}{s}$\\
      $[C_\infty] = \frac{1}{s}$}
    $\gamma_1, \ldots, \gamma_n \in \C0 (\II; \L1( \reali^+;
    \reali^n))$ are Caratheodory functions, in the sense of
    Definition~\ref{def:Cara}, and there exist positive constants
    $C_L, C_\infty$ such that
    \begin{eqnarray}
      \label{eq:15}
      \norma{\gamma_i (t,x_2) - \gamma_i (t,x_1)}
      & \leq
      & C_L \; \modulo{x_2 - x_1}
      \\
      \label{eq:20}
      \norma{\gamma_i (t,\cdot)}_{\L\infty (\reali^+; \reali)}
      +
      \tv(\gamma_i (t, \cdot); \reali^+)
      & \leq
      & C_\infty
    \end{eqnarray}
    for every $i \in \left\{1, \cdots, n\right\}$, $t \in \II$ and
    $x_1, x_2 \in \reali^+$.

  \item
    \label{hyp:B_i}\marginpar{\small$[B_L] = \frac{m}{s}$\\
      $[B_1] = n$\\$[B_\infty{}] = \frac{n}{s}$}
    $\beta_1, \ldots,
    \beta_n$ are Caratheodory functions in the sense of
    Definition~\ref{def:Cara}; for all $t$, $\beta_1 (t), \ldots,
    \beta_n (t) \in \C1 (\reali^n; \reali)$ and, for all $(t,u) \in
    \II \times \reali^n$, $\partial_{u_j} \beta_i (t, u) =
    0$ whenever $j \geq i$.  Moreover, there exist constants $B_1,
    B_\infty$ and $B_L$ such that
    \begin{eqnarray}
      \label{eq:barB-1}
      \modulo{\beta_i(t, u_1) - \beta_i(t, u _2)}
      & \le
      & B_L \; \norma{u_1 - u_2}
      \\
      \label{eq:barB-2}
      \norma{\beta_i(\cdot, 0)}_{\L1(\II; \reali)}
      & \leq
      &  B_1
      \\
      \label{eq:barB-3}
      \norma{\beta_i (\cdot, 0)}_{\L\infty (\II; \reali)}
      & \leq
      & B_\infty
      \\
      \label{eq:barB-4}
      \tv(\beta_i (\cdot, u (\cdot)); \II)
      & \leq
      & B_\infty + B_L \; \tv (u;\II)
    \end{eqnarray}
    for every $i \in \left\{1, \cdots, n\right\}$, $t \in \II$ $u_1,
    u_2 \in \reali^n$ and $u \in \BV (\II; \reali^n)$.

  \item
    \label{hyp:init-boundary}
    $u_o \in \L1 (\reali^+; \reali^n)$.
  \end{enumerate}
  \noindent Then, there exists constants $K_1$ and
  $K_{\infty}$, a positive time $\tmax$ and a constant
  $\mathcal{L}$ dependent only on $\tmax$, $\norma{u_o}_{\L1
    (\reali^+;\reali^n)}$, $\tv (u^o;
  \reali^+)$ and on the parameters
  in~\ref{hyp:g}--\ref{hyp:init-boundary} such that~\eqref{eq:2}
  admits a unique solution
  \begin{displaymath}
    u_* \in
    \C0 ([0, \tmax]; \L1 (\reali^+; \reali^n)) \,,
  \end{displaymath}
  satisfying, for all $t,t' \in [0, \tmax]$,
  \begin{equation}
    \label{eq:4}
    \begin{array}{r@{\;}c@{\;}l@{\qquad}r@{\;}c@{\;}l}
      \norma{u_* (t)}_{\L1 (\reali^+; \reali^n)}
      & \leq
      & K_1 \,,
      &
        \norma{u_* (t) - u_* (t')}_{\L1 (\reali^+; \reali^n)}
      & \leq
      & \mathcal{L} \; \modulo{t-t'} \,,
      \\
      \norma{u_* (t)}_{\L\infty (\reali^+; \reali^n)}
      & \leq
      & K_\infty \,,
      & \tv\left(u_* (t); \reali^+\right)
      & \leq
      & K_\infty \,.
    \end{array}
  \end{equation}
  Moreover, if $u_*'$ and
  $u_*''$ are the solutions to~\eqref{eq:2} corresponding to initial
  data $u_o'$ and $u_o''$ and to boundary data $\beta'$ and
  $\beta''$, then the following estimate holds for all $t \in [0,
  \tmax]$:
  \begin{equation}
    \label{eq:41}
    \norma{u_*' (t) - u_*'' (t)}_{\L1 (\reali^+; \reali^n)}
    \leq
    \mathcal{L}
    \left(
      \norma{u_o' - u_o''}_{\L1 (\reali^+; \reali^n)}
      +
      \norma{\beta'-\beta''}_{\L\infty ([0,\tmax]\times [0, K_\infty]^n; \reali^n)}
    \right) \,.
  \end{equation}
\end{theorem}

\noindent The proof is detailed in \S~\ref{sec:ProofsA}. In
particular, lower bounds for $K_1$ and $K_\infty$ are
in~\eqref{eq:constant-C}. Above, \emph{solutions} to~\eqref{eq:2} are
intended essentially in the sense of Definition~\ref{def:scalar}: note
indeed that for each $i=1, \ldots, n$, problem~\eqref{eq:2} fits
into~\eqref{eq:3}, refer to~\eqref{eq:19} for the details.

\begin{remark}
  \label{rem:remark}The assumptions in Theorem~\ref{thm:general} do
  not rule out a finite time blow-up in $u$. In fact, consider the
  IBVP
  \begin{displaymath}
    \left\{
      \begin{array}{l}
        \partial_t u + \partial_x u = \int_{\reali^+}u (t,\xi) \d\xi\; u
        \\
        u (0,x) = e^{-x}
        \\
        u (t,0) = 0
      \end{array}
    \right.
  \end{displaymath}
  that clearly fits into~\ref{hyp:g}--\ref{hyp:init-boundary} with
  $\alpha[w] = \int_{\reali^+}w (\xi)\d\xi$ and $\II = \reali^+$. Its
  (strong) solution for
  $(t,x) \in \left[0, \ln 2\right[ \times \reali^+$ is
  $u (t,x) = \dfrac{e^{t-x}}{2-e^t}$, which blows up as $t \to \ln 2$.
\end{remark}

Motivated by Remark~\ref{rem:remark}, we now strengthen the
assumptions in Theorem~\ref{thm:general} to ensure two properties of key
interest in the vaccination model~(\ref{eq:13})--(\ref{eq:IC+BDY}),
namely that the solution $u$ attains positive values and that it is
defined on all $\II$.

\begin{corollary}
  \label{cor:posGlog}
  Assume that, besides all
  assumptions~\ref{hyp:g}--\ref{hyp:init-boundary} in
  Theorem~\ref{thm:general}, we also have that
  \begin{enumerate}[label=\textbf{(POS)}, align=left]
  \item \label{it:Pos} For all $i$, $u^o_i \geq 0$ and
    $\beta_i \geq 0$.
  \end{enumerate}
  \begin{enumerate}[label=\textbf{(NEG)}, align=left]
  \item \label{it:Neg} For all $u \in \L1 (\reali^+; (\reali^+)^n)$,
    $t \in \II$ and $x \in \reali^+$,
    $\sum_{i=1}^n \left(\alpha_i [u] (x) + \gamma_i (t,x)\right) \cdot
    u (x) \leq 0$.
  \end{enumerate}
  \begin{enumerate}[label=\textbf{(EQ)}, align=left]
  \item \label{it:Eq} For all $i,j=1, \ldots, n$, $t \in \II$ and
    $x \in \reali^+$, $g_i (t,x) = g_j (t,x)$.
  \end{enumerate}
  \noindent Then, each component of the solution $u_*$ constructed in
  Theorem~\ref{thm:general} attains non negative values. Moreover, $u_*$
  can be uniquely extended to all $\II$.
\end{corollary}

\noindent The proof is deferred to \S~\ref{sec:ProofsA}.

\section{Proofs}
\label{sec:P}

Let $J$ denote a (non empty) real interval.  We use throughout the
norms
\begin{equation}
  \label{eq:26}
  \begin{array}{@{}r@{\,}c@{\,}l@{\qquad\qquad}r@{\,}c@{\,}l@{}}
    \norma{f}_{\C0 (J; \L1 (\reali^+; \reali))}
    & =
    & \sup_{t \in J} \norma{f}_{\L1 (\reali^+; \reali)}\,;
    & \norma{f}_{\L1 (\reali^+; \reali)}
    & =
    & \int_{\reali^+} \modulo{f (x)} \d{x} \,;
    \\
    \norma{f}_{\C0 (J; \L\infty (\reali^+; \reali))}
    & =
    & \sup_{t \in J} \norma{f}_{\L\infty (\reali^+; \reali)} \,;
    & \norma{f}_{\L\infty (\reali^+; \reali)}
    & =
    & \supess_{x \in \reali^+} \modulo{f (x)} \,;
    \\
    &
    &
      &
    \norma{f}_{\L1\left(\reali^+; \reali^n\right)}
    & =
    & \sum_{i=1}^n \norma{f_i}_{\L1\left(\reali^+; \reali\right)} \,.
  \end{array}
\end{equation}

\subsection{Preliminary Properties of \texorpdfstring{$\BV$}{BV}
  Functions}
\label{subs:TV}

Recall the following elementary estimates on $\BV$ functions, see
also~\cite[Section~4]{MR3327926} or~\cite{AmbrosioFuscoPallara}:
\begin{eqnarray}
  \label{eq:TV1}
  \left.
  \begin{array}{@{}r@{\,}c@{\,}l@{}}
    u
    & \in
    & \BV (\reali^+; \reali)
    \\
    w
    & \in
    & \BV (\reali^+; \reali)
  \end{array}
      \right\}
    &\Rightarrow
    & \tv (u \, w)
      \leq
      \tv (u) \, \norma{w}_{\L\infty (\reali^+; \reali)}
      +
      \norma{u}_{\L\infty (\reali^+; \reali)} \, \tv (w) \,;
  \\
  \label{eq:TV2}
  \left.
  \begin{array}{r@{\,}c@{\,}l@{}}
    \phi
    & \in
    & \C{0,1} (\reali^n; \reali)
    \\
    u
    & \in
    & \BV (\reali^+; \reali^n)
  \end{array}
      \right\}
    & \Rightarrow
    & \tv (\phi\circ u) \leq \Lip (\phi) \, \tv (u) \,;
  \\
  \label{eq:TV4}
  \left.
  \begin{array}{@{}r@{\,}c@{\,}l@{}}
    u
    & \in
    & \BV (\reali^+; \reali)
    \\
    w
    & \in
    & \BV (\reali^+; \reali)
    \\
    w (x)
    & \geq
    & \check w
      >
      0
  \end{array}
      \right\}
    & \Rightarrow
    & \tv\left(\frac{u}{w}\right)
      \leq
      \frac{1}{\check w} \, \tv (u)
      +
      \frac{1}{{\check w}^2} \, \tv (w) \, \norma{u}_{\L\infty (\reali^+; \reali)} \,;
  \\
  \label{eq:TV3}
  \!\!\!\!\!\!\left.
  \begin{array}{@{}l@{}}
    u
    \in
    \L1 (J; \L1(\reali^+;\reali))
    \\
    u (t)
    \in
    \BV (\reali^+; \reali)
  \end{array}
  \right\}
    & \Rightarrow
    & \tv\left(\int_0^t u (\tau,\cdot) \, \d\tau\right)
      \leq
      \int_0^t \tv\left(u (\tau)\right) \, \d\tau \,;
  \\
  \label{eq:3TV}
  \left.
  \begin{array}{r@{\,}c@{\,}l@{}}
    u
    & \in
    & \BV (\reali^+; \reali)
    \\
    \delta
    & \in
    & \L\infty(\reali; \reali^+)
  \end{array}
      \right\}
    & \Rightarrow
    & \int_{\reali^+}
      \modulo{u\left(x + \delta (x)\right) - u (x)}  \d{x}
      \leq
      \tv (u) \, \norma{\delta}_{\L\infty (\reali^+; \reali)} \,.
\end{eqnarray}
Inequality~\eqref{eq:TV1} follows
from~\cite[Formula~(3.10)]{AmbrosioFuscoPallara}. The definition of
total variation directly implies~\eqref{eq:TV2}, \eqref{eq:TV4}
and~\eqref{eq:TV3}.  For a proof of~\eqref{eq:3TV} see for
instance~\cite[Lemma~2.3]{BressanLectureNotes}. We supplement the
estimates above with the following one.

\begin{lemma}
  \label{lem:incubo}
  Let $J = [0, \tmax]$, with $\tmax>0$. Assume that
  $u \in \L\infty (J\times J;\reali)$ is such that
  $\sup_{\tau \in J} \tv(u (\tau, \cdot); J) < +\infty$. Then, setting
  $U (t) = \int_0^t u (\tau,t) \d\tau$,
  \begin{equation}
    \label{eq:23}
    \tv (U; J)
    \leq
    \norma{u}_{\L\infty (J\times J; \reali)} \, \tmax
    +
    \int_J \tv u (\tau, \cdot) \d\tau \,.
  \end{equation}
\end{lemma}

\begin{proof}
  Fix $t_0,t_1, \ldots, t_n$ in $J$ with $t_{i-1} < t_i$ for
  $i = 1, \ldots, n$. Using~\cite[Theorem~3.27
  and~(3.24)]{AmbrosioFuscoPallara},
  \begin{eqnarray*}
    \sum_{i=1}^n \modulo{U (t_i) - U (t_{i-1})}
    & =
    & \sum_{i=1}^n
      \modulo{\int_0^{t_i} u (\tau,t_i)\d\tau
      -
      \int_0^{t_{i-1}} u (\tau,t_{i-1})\d\tau
      }
    \\
    & =
    & \sum_{i=1}^n
      \modulo{
      \int_{t_{i-1}}^{t_i} u (\tau,t_i)\d\tau
      +
      \int_0^{t_{i-1}}
      \left(
      u (\tau,t_i)
      -
      u (\tau,t_{i-1})
      \right)
      \d\tau
      }
    \\
    & \leq
    &  \sum_{i=1}^n
      \int_{t_{i-1}}^{t_i}
      \modulo{u (\tau,t_i)}
      \d\tau
      +
      \sum_{i=1}^n
      \int_0^{t_{i-1}}
      \modulo{
      u (\tau,t_i)
      -
      u (\tau,t_{i-1})
      }
      \d\tau
    \\
    & \leq
    &  \sum_{i=1}^n
      \int_{t_{i-1}}^{t_i} \norma{u}_{\L\infty (J\times J; \reali)}\d\tau
      +
      \sum_{i=1}^n
      \int_J
      \modulo{
      u (\tau,t_i)
      -
      u (\tau,t_{i-1})
      }
      \d\tau
    \\
    & \leq
    & \norma{u}_{\L\infty (J\times J; \reali)} \, \tmax
      +
      \int_J \tv u (\tau, \cdot) \d\tau \,,
  \end{eqnarray*}
  which completes the proof.
\end{proof}

\subsection{A Scalar Renewal Equation}
\label{subs:Scalar}

\marginpar{$[u] = \frac{n}{m}$\\ $[g] = \frac{m}{s}$\\
  $[m] = \frac1s$\\ $[f]=\frac{n}{ms}$\\ $[b]=\frac{n}{s}$} We
consider the following initial--boundary value problem for a linear
scalar balance law, or \emph{renewal equation}, see
also~\cite[Chapter~3]{PerthameBook}, of the form
\begin{equation}
  \label{eq:3}
  \left\{
    \begin{array}{l@{\qquad\qquad}r@{\;}c@{\;}l}
      \partial_t u
      + \partial_x \left(g (t,x) \, u\right)
      + m (t,x) \,u
      = f (t,x)
      & (t,x)
      & \in
      & J \times \reali^+ \,,
      \\
      u (0, x) = u_o (x)
      & x
      & \in
      & \reali^+ \,,
      \\
      g (t,0) \, u (t, 0+) = b (t)
      & t
      & \in
      & J \,.
    \end{array}
  \right.
\end{equation}
Let $F_1, F_\infty, G_1, G_\infty, M,\hat g, \check g$ be
positive with $\check g < \hat g$. We require the following
conditions:
\begin{enumerate}[label=\textbf{(f)}, align=left]
\item \label{hyp:(f)} $f \in \C0\left(J; \L1(\reali^+; \reali)\right)$
  and for all $t \in J$ $\left\{
    \begin{array}{@{}r@{\,}c@{\,}l@{}}
      \norma{f (t, \cdot)}_{\L1 (\reali^+;\reali)}
      & \leq
      & F_1 \,;
      \\
      \norma{f (t, \cdot)}_{\L\infty (\reali^+; \reali)}
      +
      \tv(f (t, \cdot); \reali^+)
      & \leq
      & F_\infty 
          \,.
    \end{array} \right.$
  \marginpar{{\small$[F_\infty]=\frac{n}{ms}$\\
      $[F_1] = \frac{n}{s}$}}
\end{enumerate}

\begin{enumerate}[label=\textbf{(g)}, align=left]
\item \label{hyp:(g)}
  $g \in \C{0,1} (J \times \reali^+; [\check g, \hat g])$, for
  $(t,x) \in J\times \reali^+$
  \marginpar{$[\check g] = \frac{m}{s}$\\$[G_\infty] =
    \frac{m}{s}$\\ $[G_1] = \frac1s$} $\left\{
    \begin{array}{@{}r@{\,}c@{\,}l@{\!\!}}
      \tv(g (t,\cdot); \reali^+)
      +
      \tv(g (\cdot, x); J)
      & \leq
      & G_\infty;
      \\
      \norma{\partial_x g (t, \cdot)}_{\L\infty(\reali^+; \reali)}
      +
      \tv(\partial_x g (t,\cdot); \reali^+)
      & \leq
      & G_1 \,.
    \end{array}
  \right.$
\end{enumerate}

\begin{enumerate}[label=\textbf{(m)}, align = left]
\item \label{hyp:(m)} $m$ is a Caratheodory function and for all $t
  \in J$, \marginpar{$[M]=\frac1s$\\} $\norma{m (t, \cdot)}_{\L\infty
    (\reali^+; \reali)} + \tv(m (t,\cdot); \reali^+) \leq M$;
\end{enumerate}

\begin{enumerate}[label=\textbf{(b)}, align=left]
\item \label{hyp:(b)} $b \in \BVloc (J; \reali)$.
\end{enumerate}

\noindent Above, we refer to the usual definition of
\emph{Caratheodory function}, namely:

\begin{definition}[{\cite[\textbf{(A)} and~\textbf{(B)} in
    \S~3.1]{bressan-piccoli}}]
  \label{def:Cara}
  The map $m \colon J \times \reali^+ \to \reali^m$ is a
  \emph{Caratheodory function} if
  \begin{enumerate}
  \item For all $x \in \reali^+$, the map $m_x \colon J \to \reali^m$
    defined by $m_x (t) = m (t,x)$ is measurable.

  \item For a.e.~$t \in J$, the map $m^t \colon \reali^+ \to \reali^m$
    defined by $m^t (x) = m (t,x)$ is continuous.

  \item For all compact $K \subset J \times \reali^+$, there exist
    constants $M_\infty, M_L > 0$ such that for a.e.~$t \in J $ and
    for all $x_1, x_2 \in \reali^+$,$\norma{m (t,x)} \leq M_\infty$
    and
    $\norma{m (t,x_2) - m (t,x_1)} \leq M_L \cdot \modulo{x_2 - x_1}$.
  \end{enumerate}
\end{definition}


Recall the following definition of solution to~\eqref{eq:3}, see
also~\cite{BardosLerouxNedelec, BressanLectureNotes, Kruzkov,
  PerthameBook, 2017arXiv170509109R, SerreII}.

\begin{definition}
  \label{def:scalar}
  Assume that~\textup{\ref{hyp:(f)}}, \textup{\ref{hyp:(g)}},
  \textup{\ref{hyp:(m)}} and~\textup{\ref{hyp:(b)}} hold. Choose an
  initial datum $u_o \in \L1 (\reali^+; \reali)$. The function
  $u \in \C0 \left(J; \L1 (\reali^+; \reali)\right)$ is a
  \emph{solution} to~\eqref{eq:3} if
  \begin{enumerate}
  \item for all
    $\phi \in \Cc1 (\pint{J} \times \pint{\reali}^+; \reali)$,
    $\int_{\reali^+} \int_J \left[ u \; \partial_t \phi + g \; u \;
      \partial_x \phi + (f - m \; u) \; \phi \right] \d{t} \d{x} = 0$;
  \item $u (0,x) = u_o (x)$ for a.e.~$x \in \reali^+$;
  \item for a.e.~$t \in J$,
    $\lim_{x \to 0+} g (t,x) \, u (t,x) = b (t)$.
  \end{enumerate}
\end{definition}

\noindent As shown below, problem~\eqref{eq:3} admits as unique
solution in the sense of Definition~\ref{def:scalar} the map
\begin{equation}
  \label{eq:12}
  \!\!\!\!
  u (t,x)
  =
  \left\{
    \begin{array}{@{}lr@{\,}c@{\,}l@{}}
      \displaystyle
      u_o (X (0;t,x)) \, \mathcal{E} (0, t, x)
      +
      \int_0^t
      f \left(\tau, X (\tau;t,x)\right)
      \, \mathcal{E} (\tau, t, x)
      \d\tau
      & x
      & >
      & \sigma (t)
      \\[12pt]
      \displaystyle
      \frac{b\left(T (0;t,x)\right)}{g\left(T (0;t,x),0\right)}
      \, \mathcal{E} (T (0;t,x), t, x)
      +
      \int_{T (0;t,x)}^t
      f \left(\tau, X (\tau;t,x)\right) \, \mathcal{E} (\tau, t, x)
      \d\tau
      & x
      & <
      & \sigma (t)
    \end{array}
  \right.
  \!\!\!\!\!\!\!\!\!\!\!
\end{equation}
where\marginpar{$[\mathcal{E}] = 1$\\}
\begin{equation}
  \label{eq:11}
  \mathcal{E} (\tau, t, x)
  =
  \exp\left[-
    \int_\tau^t \left(
      m(s,X (s;t,x))
      +
      \partial_x g \left(s,X (s;t,x)\right)
    \right)
    \d{s}\right]
\end{equation}
and, for $t_o,t \in J$, $x_o,x \in \reali^+$,
\begin{equation}
  \label{eq:9}
  t \to X (t; t_o, x_o)
  \mbox{ solves }
  \left\{
    \begin{array}{l}
      \dot x = g (t,x)
      \\
      x (t_o) = x_o
    \end{array}
  \right.
  \quad \mbox{ and } \quad
  x \to T (x; t_o, x_o)
  \mbox{ solves }
  \left\{
    \begin{array}{l}
      t' = \dfrac{1}{g (t,x)}
      \\
      t (x_o) = t_o
    \end{array}
  \right.
\end{equation}
and we also set $\sigma (t) = X (t; 0, 0)$,
$\Sigma (x) = T (x; 0, 0)$.

\begin{lemma}
  \label{lem:E}
  Let~\textup{\ref{hyp:(g)}} and~\textup{\ref{hyp:(m)}} hold.  Then,
  $\mathcal{E}$ defined in~\eqref{eq:11} satisfies the following
  estimates, for $x \in \reali^+$ and $\tau, t \in J$ with
  $\tau \leq t$:
  \begin{eqnarray}
    \label{eq:17}
    \mathcal{E} (\tau, t, x)
    & \leq
    & e^{(G_1 + M) (t-\tau)} \,,
    \\
    \label{eq:18}
    \tv (\mathcal{E} (\tau, t, \cdot); \reali^+)
    & \leq
    & (G_1 + M) \,
      (t-\tau) \, e^{(G_1 + M ) (t-\tau)} \,,
    \\
    \label{eq:tv-E-vertical}
    \tv \left(\mathcal{E} (\tau, \cdot, x); [0, t]\right)
    & \leq
    & (G_1 + M) \, (t - \tau) \, e^{(G_1 + M) (t-\tau)}
      \,,
    \\
    \label{eq:tv-E-first-variable}
    \tv \left(\mathcal{E} (\cdot, t, x); [0, t]\right)
    & \le
    &
      \left(G_1 + M \right)\, t\, e^{(G_1 + M) \,t}\,.
  \end{eqnarray}
\end{lemma}

\begin{proof}
  The bound~\eqref{eq:17} directly follows
  from~\ref{hyp:(g)},~\ref{hyp:(m)}, and~\eqref{eq:11}. Consider the
  total variation estimate~\eqref{eq:18}. For $\tau \leq t$,
  by~\eqref{eq:TV2} we have
  \begin{eqnarray*}
    \tv \left(\mathcal{E} (\tau, t, \cdot); \reali^+\right)
    & \leq
    & e^{(G_1 + M) (t-\tau)}
      \int_\tau^t
      \left(
      \tv\left(m (s, \cdot); \reali^+\right)
      +
      \tv\left(\partial_x g(s, \cdot); \reali^+\right)
      \right)
      \d{s}
  \end{eqnarray*}
  which implies~\eqref{eq:18}. Now, consider the total variation
  estimate~(\ref{eq:tv-E-vertical}). For $\tau \le t$ we deduce
  \begin{flalign*}
    & \tv \left(\mathcal{E} (\tau, \cdot, x); [0,t]\right) &
    [\mbox{by~\eqref{eq:TV2}}]
    \\
    \leq \; & e^{(G_1 + M) (t-\tau)} \tv \left(\int_\tau^\cdot
      \left[m\left(s, X\left(s; \cdot, x\right)\right) + \partial_x
        g\left(s, X\left(s; \cdot, x\right)\right)\right] \d s\right)
    & [\mbox{by~\eqref{eq:23}}]
    \\
    \leq \; & e^{(G_1 + M) (t-\tau)}
    \left[\norma{m}_{\L\infty\left([0,t] \times \reali^+;
          \reali\right)} + \norma{\partial_x g} _{\L\infty\left([0,t]
          \times \reali^+; \reali\right)}\right] (t - \tau)
    \\
    & \quad + e^{(G_1 + M) (t-\tau)} \left[\sup_{s \in [0, t]} \tv
      \left( m\left(s, X\left(s; \cdot, x\right)\right) + \partial_x g
        \left(s, X\left(s; \cdot, x\right)\right)\right)\right] (t -
    \tau)
  \end{flalign*}
  and using~\ref{hyp:(g)} and~\ref{hyp:(m)} we
  deduce~(\ref{eq:tv-E-vertical}).  Finally, consider the
  estimate~(\ref{eq:tv-E-first-variable}).  We have
  \begin{flalign*}
    & \tv \left(\mathcal{E} (\cdot, t, x); [0,t]\right) &
    [\mbox{by~\eqref{eq:TV2}}]
    \\
    \leq \; & e^{(G_1 + M)\, t} \tv \left(\int_\cdot^t \left[m\left(s,
          X\left(s; t, x\right)\right) + \partial_x g\left(s,
          X\left(s; t, x\right)\right)\right] \d s\right) & [\mbox{by
      the definition of}\tv]
    \\
    \leq \; & e^{(G_1 + M) \,t} \left[\norma{m}_{\L\infty\left([0,t]
          \times \reali^+; \reali\right)} + \norma{\partial_x g}
      _{\L\infty\left([0,t] \times \reali^+; \reali\right)}\right] t,
  \end{flalign*}
  concluding the proof.
\end{proof}

The following Lemma summarizes various properties of the solution
to~\eqref{eq:3}, see also~\cite{PerthameBook}.

\begin{lemma}
  \label{lem:stability}
  Let~\textup{\ref{hyp:(f)}}, \textup{\ref{hyp:(g)}}
  and~\textup{\ref{hyp:(m)}} hold. Then, with reference to the scalar
  problem~\eqref{eq:3},
  \begin{enumerate}[label=\textup{\textbf{(SP.\arabic*)}}, align=left]

  \item \label{it:S:1} For any
    $u_o \in (\L1 \cap \BV)(\reali^+; \reali)$ and for any $b$
    satisfying~\ref{hyp:(b)}, the map
    $u \colon J \times \reali^+ \to \reali$ defined by~\eqref{eq:12}
    solves~\eqref{eq:3} in the sense of Definition~\ref{def:scalar}.

  \item \label{it:S:2} For every $t \in J$, the following \emph{a
      priori} estimates hold:
    \begin{eqnarray*}
      \sup_{\tau \in [0, t]}
      \norma{u (\tau)}_{\L\infty (\reali^+; \reali)}
      & \!\!\leq\!\!
      & \left(
        \norma{u_o}_{\L\infty (\reali^+; \reali)}
        +
        \frac{1}{\check g} \, \norma{b}_{\L\infty ([0,t];\reali)}
        +
        F_\infty \, t
        \right)
        \,  e^{(G_1 + M) t} \,,
      \\
      \sup_{\tau \in [0, t]}
      \norma{u (\tau)}_{\L1 (\reali^+; \reali)}
      & \!\!\leq\!\!
      & \left(
        \norma{u_o}_{\L1 (\reali^+; \reali)}
        +
        \norma{b}_{\L1 ([0,t];\reali)}
        +
        F_1 \, t
        \right)
        e^{M t} \,.
    \end{eqnarray*}

  \item \label{it:S:3} For every $t \in J$, the following total
    variation estimate holds
    \begin{equation}
      \label{eq:tvx}
      \!\!\!\!\!\!\!\!\!\!\!\!
      \tv(u (t); \reali^+)
      \leq
      \mathcal{H}(t) \!
      \left( \!
        F_\infty t
        {+}
        \frac{\norma{b}_{\L\infty ([0,t]; \reali)}
          {+}
          \tv (b;[0,t])
        }{\check g}
        {+}
        \norma{u_o}_{\L\infty (\reali^+; \reali)}
        {+}
        \tv (u_o; \reali^+)
        \! \right)
    \end{equation}
    where $\mathcal{H} (t)$ is a non decreasing continuous function of
    $t$, depending also on $\check g, G_1, G_\infty$ and $M$,
    satisfying $\mathcal H(0) \le 5 + G_\infty / \check g$.

  \item \label{it:S:4} Fix $t \in J$ and $x \in \reali^+$.  If
    $x > \sigma (t)$, then
    \begin{equation}
      \label{eq:tv-vertical}
      \begin{split}
        \tv \left(u(\cdot, x); [0, t]\right) \le & \left[ \tv (u_o;
          \reali^+) + 2 (G_1 + M) \norma{u_o}_{\L\infty\left(\reali^+;
              \reali\right)} t\right] e^{\left(G_1 + M \right)\, t}
        \\
        & \quad + 4 \left[ 1 + \left(G_1 + M \right) t \, \right]
        F_\infty \, t\, e^{(G_1 + M)\, t}.
      \end{split}
    \end{equation}
    If $x < \sigma(t)$, then
    \begin{equation}
      \label{eq:tv-vertical-general}
      \begin{split}
        \tv\left(u (\cdot,x); [0,t]\right) \leq & \left[\tv (u_o;
          \reali^+) + \frac{1}{\check g}\tv \left(b(\cdot); [0,
            t]\right)\right] e^{(G_1 +M)\, t}
        \\
        & + 2 \left[1 + (G_1 + M) t\right]
        \norma{u_o}_{\L\infty\left(\reali^+; \reali\right)} e^{(G_1 +
          M) t}
        \\
        & + \frac{1}{\check g} \left[2 + 3 (G_1 + M) t +
          \frac{G_\infty}{\check g}\right]
        \norma{b}_{\L\infty\left([0, t]; \reali\right)} \, e^{(G_1 +
          M) t}
        \\
        & + 2 \left(7 + 6 (G_1 + M) t\right) F_\infty \, t \, e^{(G_1
          + M)t}\, .
      \end{split}
    \end{equation}

  \item \label{it:S:4:new} Fix a positive $W$. For any
    $w \in (\C1 \cap \BV) (J;[-W,W])$,
    \begin{eqnarray*}
      \tv \left(\int_{\reali^+} w (\cdot, x) \, u (\cdot, x) \d{x}; [0,t]\right)
      & \leq
      & \norma{u}_{\L\infty ([0,t]\times\reali^+; \reali)}
        \int_{\reali^+} \tv\left(w (\cdot,x); [0,t]\right) \d{x}
      \\
      &
      &  +
        W \int_{\reali^+} \tv\left(u (\cdot,x); [0,t]\right) \d{x}
    \end{eqnarray*}

  \item \label{it:S:5} For every $t \in J$, there exists a positive
    $\mathcal L$ dependent on $\norma{u_o}_{\L1 (\reali^+; \reali)}$
    and on the constants in~\ref{hyp:(f)}, \ref{hyp:(g)},
    \ref{hyp:(m)} and~\ref{hyp:(b)}, such that, for
    $t', t'' \in [0,t]$,\marginpar{$[\mathcal{L}] = \frac{n}{s}$}
    \begin{equation}
      \label{eq:lip-dependence}
      \norma{u(t') - u(t'')}_{\L1(\reali^+; \reali)}
      \leq
      \mathcal{L} \; \modulo{t'' - t'}.
    \end{equation}

  \item \label{it:S:6} If $u_o \geq 0$, $f \geq 0$ and $b \geq 0$,
    then $u (t) \geq 0$ for all $t$.
  \end{enumerate}
\end{lemma}

\begin{remark} The boundedness of the space variation of $f$ required
  in~\ref{hyp:(f)} is necessary. Indeed, consider Problem~\eqref{eq:3}
  with $g (t,x) = 1$, $m (t,x) = 0$, $f (t,x) = \sin \dfrac{1}{x-t}$,
  $u_o (x) = 0$ and $b (t) = 0$.  The solution is
  $u (t,x) = t \, \sin \dfrac{1}{x-t}$ which has unbounded total
  variation in space for all $t>0$.
\end{remark}

\begin{proofof}{Lemma~\ref{lem:stability}}
  We prove the different items separately.

  \paragraph{\ref{it:S:1}:}
  A standard integration along characteristics is sufficient to prove
  it.

  \paragraph{\ref{it:S:2}:}
  These bounds are an immediate consequence of~\ref{hyp:(f)},
  \ref{hyp:(g)}, \ref{hyp:(m)} and~\eqref{eq:12}.

  \paragraph{\ref{it:S:3}:}

  We clearly have
  \begin{equation}
    \label{eq:16}
    \!\!\!\!\!\!
    \tv \! \left(u (t)\right)
    {=}
    \tv \! \left(u (t, \cdot), \left[0, \sigma (t)\right[\right)
    {+}
    \modulo{u\left(t, \sigma (t)+\right) - u\left(t, \sigma (t)-\right)}
    {+}
    \tv \! \left(u (t, \cdot), \left]\sigma (t), +\infty\right[\right)
    \!\!\!\!\!\!
  \end{equation}
  and we estimate the three terms in the right hand side
  of~\eqref{eq:16} separately. Begin with the first one, using the
  second expression in~\eqref{eq:12}:
  \begin{eqnarray*}
    &
    & \tv\left(u (t, \cdot), \left[0, \sigma (t)\right[\right)
    \\
    & \leq
    & \tv \left(\frac{b (\cdot)}{g (\cdot , 0)}; [0,t]\right) e^{(G_1 + M) t}
      +
      (G_1 + M) \, t \, \frac{\norma{b}_{\L\infty ([0,t]; \reali)}}{\check g}
      e^{(G_1 + M) t}
    \\
    &
    & + \int_0^t \tv \left(
      f \left(\tau, X (\tau; t, \cdot)\right) \,
      \mathcal{E} (\tau, t, \cdot)
      \right) \d\tau
      +
      \norma{f}_{\L\infty ([0,t]\times \reali^+; \reali)} \, e^{(G_1 + M)t} \,
      \tv\left(T (0; t, \cdot)\right)
    \\
    & \leq
    & \left(
      \dfrac{1}{\check g} \, \tv (b; [0,t])
      +
      \dfrac{1}{{\check g}^2} \, G_\infty \, \norma{b}_{\L\infty ([0,t]; \reali)}
      +
      (G_1 + M) \, t \, \frac{\norma{b}_{\L\infty ([0,t]; \reali)}}{\check g}
      \right)
      e^{(G_1 + M) t}
    \\
    &
    & +
      (G_1 + M)\, F_\infty \,
      t^2 \, e^{(G_1 + M)t}
      + \sup_{t \in J} \tv (f (t, \cdot))  \, t \; e^{(G_1 + M)t}
      + F_\infty \, t \, e^{(G_1 +M)t}
    \\
    & \leq
    & \dfrac{1}{\check g}
      \left(
      \tv (b;[0,t])
      +
      \left(\dfrac{G_\infty}{\check g}
      + (G_1 + M) \, t\right) \, \norma{b}_{\L\infty ([0,t]; \reali)}
      \right)
      e^{(G_1 + M) t}
    \\
    &
    & +
      \left( (G_1 + M)\, t + 2\right) \,
      F_\infty  \, t \, e^{(G_1 + M)t} \,.
  \end{eqnarray*}
  Concerning the second term in~\eqref{eq:16}, the following rough
  estimate is sufficient for later use:
  \begin{flalign*}
    \modulo{u(t, \sigma (t)+) - u(t, \sigma (t)-)} \leq \; & 2
    \,\norma{u (t)}_{\L\infty (\reali^+; \reali)} &
    \!\!\!\![\mbox{By~\ref{it:S:2}}]
    \\
    \leq \; & 2 \left( \norma{u_o}_{\L\infty (\reali^+; \reali)} +
      \frac{1}{\check g} \, \norma{b}_{\L\infty ([0,t];\reali)} +
      F_\infty \, t \right) e^{(G_1 + M) t}.
  \end{flalign*}
  The latter term in~\eqref{eq:16} reads
  \begin{eqnarray*}
    &
    & \tv\left(u (t, \cdot), \left]\sigma (t), +\infty\right[\right)
    \\
    & \leq
    & (G_1 + M)
      \norma{u_o}_{\L\infty (\reali^+; \reali)} \, t \, e^{(G_1 + M)t}
      +
      \tv (u_o) \, e^{(G_1 + M)t}
    \\
    &
    & \qquad
      + \int_0^t \tv\left(
      f\left(\tau, X (\tau; t, \cdot)\right) \, \mathcal{E} (\tau, t, \cdot)
      \right) \d{\tau}
    \\
    & \leq
    & \left((G_1 + M) \, \norma{u_o}_{\L\infty (\reali^+; \reali)} \, t
      + \tv (u_o)\right)
      e^{(G_1 + M)t}
    \\
    &
    & \qquad
      +
      (G_1 + M) \, \norma{f}_{\L\infty ([0,t]\times\reali^+; \reali)} \,
      t^2 \, e^{(G_1 + M)t}
      + \sup_{t \in J} \tv (f (t, \cdot))  \, t \; e^{(G_1 + M)t}
    \\
    & \leq
    & \left(
      (G_1 + M) \,
      \left(
      \norma{u_o}_{\L\infty (\reali^+; \reali)}
      +
      F_\infty \, t
      \right)
      t
      + \tv (u_o)
      + F_\infty \, t
      \right)
      e^{(G_1 + M)t} \,.
  \end{eqnarray*}
  Using now~\eqref{eq:16},
  \begin{eqnarray*}
    &
    & \tv \left(u (t, \cdot); \reali^+\right)
    \\
    & \leq
    & \Bigl[
      \frac{1}{\check g} \tv (b;[0,t])
      + \frac{1}{\check g}
      \left( \dfrac{G_\infty}{\check g} + (G_1 + M) \, t\right) \,
      \norma{b}_{\L\infty ([0,t]; \reali)}
    \\
    &
    & +
      \left( (G_1 + M)\, t + 2\right) \,
      F_\infty \, t
      +
      2 \left(
      \norma{u_o}_{\L\infty (\reali^+; \reali)}
      +
      \frac{1}{\check g} \, \norma{b}_{\L\infty ([0,t];\reali)}
      +
      F_\infty \, t
      \right)
    \\
    &
    & +
      (G_1 + M) \, t \,
      \left(
      \norma{u_o}_{\L\infty (\reali^+; \reali)}
      + F_\infty \, t
      \right)
      + \tv (u_o)
      + F_\infty  \, t
      \Bigr]  e^{(G_1 + M)t}
    \\
    & \leq
    & \Bigl[ \frac{1}{\check g}
      \left( 2 + \dfrac{G_\infty}{\check g} + (G_1 + M) \, t\right) \,
      \norma{b}_{\L\infty ([0,t]; \reali)}
      + \frac{1}{\check g} \tv (b;[0,t])
      +
      \left((G_1 + M) t + 5\right)
      \, F_\infty \, t
    \\
    &
    & +
      \left(2 + (G_1 + M) \, t\right)
      \norma{u_o}_{\L\infty (\reali^+; \reali)}
      +
      \tv (u_o)
      \Bigr] e^{(G_1 + M)t}
    \\
    & \leq
    & \mathcal{H}(t)
      \left(
      F_\infty t
      +
      \frac{1}{\check g} \norma{b}_{\L\infty ([0,t]; \reali)}
      +
      \frac{1}{\check g} \tv (b;[0,t])
      +
      \norma{u_o}_{\L\infty (\reali^+; \reali)}
      +
      \tv (u_o)
      \right)
  \end{eqnarray*}
  we prove~\eqref{eq:tvx}.

\paragraph{\ref{it:S:4}:}
First, in view of an application of Lemma~\ref{lem:incubo}, compute
\begin{flalign*}
  & \tv\left( f\left(\tau,X (\tau;\cdot,x)\right) \mathcal{E}
    (\tau,\cdot,x); [0,t] \right) & [\mbox{Use~\eqref{eq:TV1}]}
  \\
  \leq & \tv\left( f\left(\tau,X (\tau;\cdot,x)\right) ; [0,t] \right)
  \norma{\mathcal{E} (\tau,\cdot,x)}_{\L\infty ([0,t];\reali)} &
  [\mbox{Use~\eqref{eq:TV2} and~\eqref{eq:17}}]
  \\
  & + \norma{f\left(\tau,X (\tau;\cdot,x)\right)}_{\L\infty
    ([0,t];\reali)} \tv\left( \mathcal{E} (\tau,\cdot,x); [0,t]
  \right) & [\mbox{Use~\eqref{eq:tv-E-vertical}}]
  \\
  \leq & e^{(G_1 + M)t} \tv \left(f (\tau, \cdot); \reali^+\right)
  \\
  & + 2 (G_1 + M) t \, e^{(G_1 + M) t} \, \norma{f}_{\L\infty
    ([0,t]\times\reali; \reali)} \,.
\end{flalign*}
Thus, using~\ref{hyp:(f)}, we deduce that
\begin{equation}
  \label{eq:tv-FE-vert}
  \tv\left(
    f\left(\tau,X (\tau;\cdot,x)\right) \mathcal{E} (\tau,\cdot,x); [0,t]
  \right)
  \le
  \left(1 + 2 (G_1 + M) t\right) \, F_\infty \, e^{(G_1 + M)t} \,.
\end{equation}
First consider the simple case $x > \sigma(t)$; we have
\begin{flalign*}
  & \tv\left(u (\cdot,x); [0,t]\right) & [\mbox{Use~(\ref{eq:12})}]
  \\
  \leq & \tv\left( u_o\left(X (0;\cdot,x)\right) \mathcal{E}
    (0,\cdot,x); [0,t] \right) & [\mbox{Use~\eqref{eq:TV1}}]
  \\
  & + \tv \left( \int_0^\cdot
    f\left(\tau,X\left(\tau;\cdot,x\right)\right) \mathcal{E}
    (\tau,\cdot,x) \d\tau; [0,t] \right) & [\mbox{use
    Lemma~\ref{lem:incubo}}]
  \\
  \leq & \tv\left( u_o\left(X (0;\cdot,x)\right) ; [0,t] \right)
  \norma{\mathcal{E} (0,\cdot,x)}_{\L\infty ([0,t]; \reali^+)} &
  [\mbox{Use~(\ref{eq:17})}]
  \\
  & + \norma{u_o\left(X (0;\cdot,x)\right)}_{\L\infty ([0,t];
    \reali^+)} \tv\left( \mathcal{E} (0,\cdot,x); [0,t] \right) &
  [\mbox{Use~(\ref{eq:tv-E-vertical})}]
  \\
  & + 2 \sup_{\tau \in [0, t]} \tv \left( f\left(\tau,X
      (\tau;\cdot,x)\right) \mathcal{E} (\tau,\cdot,x); [0, t] \right)
  t & [\mbox{Use~(\ref{eq:tv-FE-vert})}]
  \\
  & + 2 \norma{f}_{\L\infty\left([0, t] \times \reali^+;
      \reali\right)} \norma{\mathcal E}_{\L\infty\left(\left[0,
        t\right]^2 \times \reali^+; \reali\right)}t &
  [\mbox{Use~\ref{hyp:(f)} and~(\ref{eq:17})}]
  \\
  \leq & \tv \left(u_o; \reali^+\right) e^{(G_1 + M)\, t} + 2 (G_1 +
  M) \norma{u_o}_{\L\infty\left(\reali^+; \reali\right)} \, t \,
  e^{(G_1 + M) t}
  \\
  & + 2 \left(1 + 2 (G_1 + M) t \right) \, F_\infty \, t \, e^{(G_1 +
    M)t} + 2 \, F_\infty \, t \, e^{(G_1 + M)t}\,.
\end{flalign*}
Now consider the case $x < \sigma(t)$, i.e. $\Sigma(x) < t$.  We
clearly have
\begin{equation}
  \label{eq:tv-vert-generic-1}
  \tv\left(u (\cdot,x); [0,t]\right)
  \le
  \tv\left(u (\cdot,x); [0,\Sigma(x)]\right)
  + \tv\left(u (\cdot,x); [\Sigma(x),t]\right)
  + 2 \norma{u}_{\L\infty\left([0,t] \times \reali^+; \reali\right)}.
\end{equation}
Let us estimate the second term
$\tv\left(u (\cdot,x); [\Sigma(x),t]\right)$.
\begin{flalign*}
  & \tv\left(u (\cdot,x); [\Sigma(x),t]\right) &
  [\mbox{Use~\eqref{eq:12}}]
  \\
  \leq & \tv\left( \frac{b\left(T (0;\cdot,x)\right)}{g\left(T(0;
        \cdot, x), 0\right)} \, \mathcal{E} (T(0; \cdot, x),\cdot,x);
    [\Sigma(x),t] \right) & [\mbox{Use~\eqref{eq:TV1}, ~\eqref{eq:TV4}}]
  \\
  & + \tv \left( \int_{T\left(0; \cdot, x\right)}^\cdot
    f\left(\tau,X\left(\tau;\cdot,x\right)\right) \mathcal{E}
    (\tau,\cdot,x) \d\tau; [\Sigma(x),t] \right) & [\mbox{Use
    Lemma~\ref{lem:incubo}}]
  \\
  \le & \frac{1}{\check g} \tv \left(b\left(T (0;\cdot,x)\right);
    [\Sigma(x),t]\right) \norma{\mathcal{E} (T(0; \cdot, x),\cdot,x)}
  _{\L\infty([0, t]; \reali)} & [\mbox{Use~(\ref{eq:17})}]
  \\
  & + \frac{1}{{\check g}^2} \tv (g(T (0;\cdot,x), 0); [\Sigma(x),t])
  \norma{b (T(0; \cdot, x)) \, \mathcal{E} (T(0; \cdot, x),\cdot,x)}
  _{\L\infty([0, t]; \reali)} \!\!\!\!\!\!\!\!\!
  & [\mbox{Use~\ref{hyp:(g)}, (\ref{eq:17})}]
  \\
  & + \frac{1}{\check g} \, \norma{b (T(0; \cdot, x))} _{\L\infty([0,
    t]; \reali)} \tv\left(\mathcal{E} (T(0; \cdot, x),\cdot,x);
    [\Sigma(x), t]\right) & [\mbox{Use~(\ref{eq:tv-E-vertical}), (\ref{eq:tv-E-first-variable})}]
  \\
  & + 4 \left(t - \Sigma(x)\right) \sup_{\tau \in [\Sigma(x), t]} \tv
  \left(f\left(\tau,X\left(\tau;\cdot,x\right)\right) \mathcal{E}
    (\tau,\cdot,x); [\Sigma(x), t]\right) &
  [\mbox{Use~(\ref{eq:tv-FE-vert})}]
  \\
  & + 4 \left(t - \Sigma(x)\right)
  \norma{f\left(\cdot,X\left(\cdot;\cdot,x\right)\right) \mathcal{E}
    (\cdot,\cdot,x)}_{\L\infty\left([\Sigma(x), t]^2; \reali\right)} &
  [\mbox{Use~\ref{hyp:(f)}, (\ref{eq:17})}]
  \\
  \le & \frac{1}{\check g}\left[\tv \left(b\left( \cdot\right);
      [0,t]\right) + \left(3 (G_1 + M)t + \frac{G_\infty}{\check
        g}\right)\norma{b} _{\L\infty\left([0, t]; \reali\right)}
  \right] e^{(G_1 + M)t}
  \\
  & + 4 \left(t -\Sigma(x)\right) \left[2 + 2\, t\,\left(G_1 +
      M\right)\right] \, F_\infty \, e^{(G_1 + M)t} \,.
\end{flalign*}
Therefore, using~\ref{it:S:2}, we deduce that
\begin{eqnarray*}
  &
  & \tv\left(u (\cdot,x); [0,t]\right)\\
  & \leq
  & \left[\tv \left(u_o; \reali^+\right)
    + \frac{1}{\check g}\tv \left(b(\cdot); [0, t]\right)\right]
    e^{(G_1 +M)\, t}
    + 2 \left[1 + (G_1 + M) t\right]
    \norma{u_o}_{\L\infty\left(\reali^+; \reali\right)}
    e^{(G_1 + M) t}
  \\
  &
  & + \frac{1}{\check g} \left[2 + 3 (G_1 + M)
    t + \frac{G_\infty}{\check g}\right]
    \norma{b}_{\L\infty\left([0, t]; \reali\right)} \,
    e^{(G_1 + M) t}
    + 2 \left(7 + 6 (G_1 + M) t\right) F_\infty \, t
    \, e^{(G_1 + M)t}\, .
\end{eqnarray*}
This completes the proof of~\ref{it:S:4}.

\paragraph{\ref{it:S:4:new}:}
Using the already obtained estimates, we have:
\begin{flalign*}
  & \tv \left(\int_J w (\cdot, x) \, u (\cdot, x) \d{x}; [0,t]\right)
  \\
  \leq & \int_J \tv \left( w (\cdot, x) \, u (\cdot, x) ; [0,t]\right)
  \d{x} & [\mbox{By~\eqref{eq:TV3}}]
  \\
  \leq & \int_J \!
  \left( \! \tv\left(w (\cdot, x); [0,t]\right) \, \norma{u
      (\cdot, x)}_{\L\infty ([0,t];\reali)} {+} \norma{w
      (\cdot,x)}_{\L\infty ([0,t];\reali)} \tv\left( \!u (\cdot, x);
      [0,t]\right) \! \right) \! \d{x} & [\mbox{By~\eqref{eq:TV1}}]
  \\
  \leq & \int_J \tv\left(w (\cdot, x); [0,t]\right) \d{x}
  \norma{u}_{\L\infty ([0,t]\times J;\reali)} + W \int_J \tv\left(u
    (\cdot, x); [0,t]\right) \d{x}
\end{flalign*}

\paragraph{\ref{it:S:5}:} Fix $t',t'' \in J$ with $t' < t''$. Then,
\begin{displaymath}
  \norma{u (t'') - u (t')}_{\L1 (\reali^+; \reali)}
  =
  \int_0^{X (t''; t', 0)} \modulo{u (t'',x) - u (t',x)} \d{x}
  +
  \int_{X (t''; t', 0)}^{+\infty} \modulo{u (t'',x) - u (t',x)} \d{x}
\end{displaymath}
We estimate the two latter terms above separately:
\begin{flalign*}
  & \int_0^{X (t''; t', 0)} \modulo{u (t'',x) - u (t',x)} \d{x}
  \\
  \leq & \left( \norma{u (t'')}_{\L\infty (\reali^+; \reali)} +
    \norma{u (t')}_{\L\infty (\reali^+; \reali)} \right) X (t''; t',0)
  & [\mbox{by~\ref{it:S:2} and~\ref{hyp:(g)}]}
  \\
  \leq & 2 {\hat g} \left( \norma{u_o}_{\L\infty (\reali^+; \reali)} +
    \frac{1}{\check g} \, \norma{b}_{\L\infty ([0,t''];\reali)} +
    \norma{f}_{\L\infty ([0,t'']\times\reali^+;\reali)} t'' \right)
  (t''-t') \, e^{(G_1 + M) t''} \,.\hspace{-3cm}
\end{flalign*}
Passing to the next term,
\begin{flalign*}
  & \displaystyle \int_{X (t''; t', 0)}^{+\infty} \modulo{u (t'',x) -
    u (t',x)} \d{x}
  \\
  \leq & \displaystyle \int_{X (t''; t', 0)}^{+\infty}
  \modulo{u\left(t', X (t';t'',x)\right)} \modulo{\mathcal{E}
    (t',t'',x) - 1} \d{x} & [\mbox{Use~\eqref{eq:17}]}
  \\
  &\displaystyle + \int_{X (t''; t', 0)}^{+\infty} \modulo{u\left(t',
      X (t';t'',x)\right) - u (t',x)} \d{x} &
  [\mbox{Use~\eqref{eq:3TV}}]
  \\
  & \displaystyle + \int_{X (t''; t', 0)}^{+\infty} \int_{t'}^{t''}
  \modulo{f \left(\tau, X (\tau;t'',x)\right)} \mathcal{E}
  (\tau,t'',x) \d\tau \d{x} & [\mbox{Use~\textup{\ref{hyp:(f)}}
    and~\eqref{eq:17}]}
  \\
  \leq & \displaystyle\, \norma{u (t')}_{\L1 (\reali^+; \reali)}
  \left(e^{(G_1+M) (t''-t')}-1\right) & [\mbox{Use~\ref{it:S:2}]}
  \\
  & \displaystyle + \tv\left(u (t')\right) \norma{X (t'; t'', \cdot) -
    \cdot}_{\L\infty (\reali^+; \reali)} &
  [\mbox{Use~\textup{\ref{hyp:(g)}} and~\ref{it:S:2}}]
  \\
  & \displaystyle + \sup_{[t',t'']} \norma{f (t)}_{\L1 (\reali^+;
    \reali)} e^{(G_1 +M) (t''-t')} (t''-t') &
  [\mbox{Use~\textup{\ref{hyp:(f)}} and~\eqref{eq:17}]}
  \\
  \leq & \displaystyle \mathcal{L} \, (t''-t') \,,
\end{flalign*}
completing the proof of~\eqref{eq:lip-dependence}.

\paragraph{\ref{it:S:6}:} This bound is a direct consequence
of~\eqref{eq:12}.
\end{proofof}

\begin{lemma}
  \label{lem:stability2}
  Let~\textup{\ref{hyp:(g)}} holds.  Fix
  $u_o', u_o'' \in (\L1 \cap \BV) (\reali^+; \reali)$, $b', b''$
  satisfying~\textup{\ref{hyp:(b)}}, $m', m''$
  satisfying~\textup{\ref{hyp:(m)}}, and $f', f''$
  satisfying~\textup{\ref{hyp:(f)}}. Call $u'$ and $u''$ the solutions
  to
  \begin{equation}
    \label{eq:Two}
    \!\!\!\!\!\!\!\!\!
    \left\{
      \begin{array}{@{\,}l@{}}
        \partial_t u
        +
        \partial_x \!\left(g (t,x) \, u\right) + m' (t,x) \, u
        = f' (t,x)
        \\
        u (0, x) = u_o' (x)
        \\
        g (t,0) \, u (t, 0+) = b' (t)
      \end{array}
    \right.
    \mbox{and }
    \left\{
      \begin{array}{@{\,}l@{}}
        \partial_t u
        +
        \partial_x \!\left(g (t,x) \, u\right) + m'' (t,x) \, u
        = f'' (t,x)
        \\
        u (0, x) = u_o'' (x)
        \\
        g (t,0) \, u (t, 0+) = b'' (t) \,.
      \end{array}
    \right.
    \!\!\!\!\!\!\!\!\!
  \end{equation}
  Then,
  \begin{enumerate}[label=\textup{\textbf{(SP.\arabic*)}}, start=8]

  \item \label{it:SP8} The following stability conditions hold:
    \begin{equation}
      \label{eq:stability-linear}
      \begin{split}
        & \norma{u' (t)- u'' (t)}_{\L1 (\reali^+; \reali)}
        \\
        \le & e^{Mt} \norma{u_o' - u''_o}_{\L1 (\reali^+; \reali)} +
        e^{2 (G_1 + M) t} \left( 2 \norma{f' - f''} _{\L1(J \times
            \reali^+; \reali)} + \norma{b' - b''}_{\L1([0, t];\reali)}
        \right)
        \\
        & \quad + e ^{(2G_1 + M) t} \left[\norma{u_o''}_{\L1
            (\reali^+; \reali)} + 2 t F_1 + \norma{b''}_{\L1([0, t];
            \reali)}\right]\, t\, \norma{m' - m''}_{\L\infty([0,t]
          \times \reali^+; \reali)}
      \end{split}
    \end{equation}
    and
    \begin{equation}
      \label{eq:stability-linear-2}
      \begin{split}
        \!\!\!\!\!\!
        & \norma{u' (t)- u'' (t)}_{\L1 (\reali^+; \reali)}
        \\
        \!\!\!\!\!\!
        \le & e^{Mt} \norma{u_o' - u''_o}_{\L1 (\reali^+; \reali)} +
        e^{2 (G_1 + M) t} \left( 2 \norma{f' - f''} _{\L1(J \times
            \reali^+; \reali)} + \norma{b' - b''}_{\L1([0, t];\reali)}
        \right)
        \\
        \!\!\!\!\!\!
        & \quad + e ^{(2G_1 + M) t} \left[\norma{u_o''}_{\L\infty
            (\reali^+; \reali)} + 2 t F_\infty
          + \frac{\norma{b''}_{\L\infty([0, t];
              \reali)}}{\check g}\right]\,
        \norma{m' - m''}_{\L1([0,t]\times\reali^+; \reali)}.
      \end{split}
    \end{equation}

  \item \label{it:SP9} The following monotonicity property holds:
    \begin{equation}
      \label{eq:mono}
      \left.
        \begin{array}{@{}r@{\;}c@{\;}l@{\qquad\forall}c@{\,}c@{\,}l@{}}
          f' (t,x)
          & \leq
          & f'' (t,x)
          & (t,x)
          & \in
          & J \times\reali^+
          \\
          u_o' (x)
          & \leq
          & u_o'' (x)
          & x
          & \in
          & \reali^+
          \\
          b' (t)
          & \leq
          & b'' (t)
          & t
          & \in
          & J
        \end{array}\!
      \right\}
      \Rightarrow
      u' (t,x) \leq  u'' (t,x)
      \quad \forall (t,x) \in J \times \reali^+ \,.
    \end{equation}

  \item \label{it:SP10} If $\bar x > 0$ and $\sigma(t) < \bar x$, then
    \begin{equation}
      \label{eq:vertical-stability}
      \begin{split}
        & \norma{u'(\cdot, \bar x) - u''(\cdot, \bar x)} _{\L1([0,t];
          \reali)}
        \\
        \le & e^{(G_1 + M) t} \, t\, \left[ e^{G_1 t}
          \norma{u'_o}_{\L1\left(\reali^+; \reali\right)} + t^2
          F_\infty\right] \norma{m'- m''} _{\L\infty([0,t] \times
          \reali^+; \reali)}
        \\
        & \quad + e^{Mt} \norma{u'_o - u''_o}_{\L1\left(\reali^+;
            \reali\right)} + e^{\left(2 G_1 + M\right) t} \norma{f' -
          f''}_{\L1\left([0,t] \times \reali^+; \reali\right)}.
      \end{split}
    \end{equation}
    and
    \begin{equation}
      \label{eq:vertical-stability-2}
      \begin{split}
        & \norma{u'(\cdot, \bar x) - u''(\cdot, \bar x)} _{\L1([0,t];
          \reali)}
        \\
        \le & {\frac{e^{(G_1 + M) t}}{G_\infty} \, \left[ e^{G_1 t}
          \norma{u'_o}_{\L\infty\left(\reali^+; \reali\right)} + t
          F_\infty\right]
        \norma{m' - m''} _{\L1([0,t]\times \reali^+; \reali)}}
      \\
        & \quad + e^{Mt} \norma{u'_o - u''_o}_{\L1\left(\reali^+;
            \reali\right)} + e^{\left(2 G_1 + M\right) t} \norma{f' -
          f''}_{\L1\left([0,t] \times \reali^+; \reali\right)}.
      \end{split}
    \end{equation}
  \end{enumerate}
\end{lemma}

\begin{proof}
  The stability bounds~(\ref{eq:stability-linear})
  and~(\ref{eq:stability-linear-2}) can be easily proved
  using the explicit formula~(\ref{eq:12}) and the estimates of
  Lemma~\ref{lem:E}.  Also the monotonicity property~\eqref{eq:mono}
  directly follows from~\eqref{eq:12}.

  We pass to the proof of the estimate~(\ref{eq:vertical-stability}).
  The proof of the estimate~(\ref{eq:vertical-stability-2}) is completely
  analogous.
  Using~(\ref{eq:12}) with the condition $\sigma(t) < \bar x$ we
  deduce that
  \begin{align*}
    \norma{u'(\cdot, \bar x) - u''(\cdot, \bar x)}
    _{\L1([0,t]; \reali)}
    & \le
      \int_0^{t} \modulo{u_o'\left(X\left(0; s, \bar x\right)\right)}
      \modulo{\mathcal E' \left(0, s, \bar x\right)
      - \mathcal E'' \left(0, s, \bar x\right)} \d s
    \\
    & \quad +
      \int_0^{t} \modulo{u_o'\left(X\left(0; s, \bar x\right)\right)
      - u_o''\left(X\left(0; s, \bar x\right)\right)}
      \modulo{\mathcal E'' \left(0, s, \bar x\right)} \d s
    \\
    & \quad +
      \int_0^{t} \int_0^s
      \modulo{f'\left(\tau, X\left(\tau; s, \bar x\right)\right)}
      \modulo{\mathcal E' \left(\tau, s, \bar x\right)
      - \mathcal E'' \left(\tau, s, \bar x\right)} \d \tau\, \d s
    \\
    & \quad +
      \int_0^{t} \!\! \int_0^s \!
      \modulo{f'\left(\tau, X\left(\tau; s, \bar x\right)\right)
      {-} f''\left(\tau, X\left(\tau; s, \bar x\right)\right)}
      \modulo{\mathcal E'' \left(\tau, s, \bar x\right)} \d \tau \d s.
  \end{align*}
  We now estimate all the terms in the previous inequality.
  Using~(\ref{eq:11}), \ref{hyp:(m)} and~\ref{hyp:(g)} we have
  \begin{align*}
    & \int_0^{t} \modulo{u_o'\left(X\left(0; s, \bar x\right)\right)}
      \modulo{\mathcal E' \left(0, s, \bar x\right)
      - \mathcal E'' \left(0, s, \bar x\right)} \d s
    \\
    \le \,
    & e^{(G_1 + M) t} \, t\, \norma{m'- m''}
      _{\L\infty([0,t] \times \reali^+; \reali)}
      \int_0^t \modulo{u_o'\left(X\left(0; s, \bar x\right)\right)}
      \d s
    \\
    \le\,
    &
      e^{\left(2 G_1 + M\right) t}  \, t\, \norma{m'- m''}
      _{\L\infty([0,t] \times \reali^+; \reali)}
      \norma{u_o'}_{\L1(\reali^+; \reali)}.
  \end{align*}
  Using~(\ref{eq:11}), and~\ref{hyp:(m)}, we deduce that
  \begin{equation*}
    \int_0^{t} \modulo{u_o'\left(X\left(0; s, \bar x\right)\right)
      - u_o''\left(X\left(0; s, \bar x\right)\right)}
    \modulo{\mathcal E'' \left(0, s, \bar x\right)} \d s
    \le\,
    e^{Mt} \norma{u'_o - u''_o}_{\L1\left(\reali^+; \reali\right)}.
  \end{equation*}
  Using~(\ref{eq:11}), \ref{hyp:(m)}, \ref{hyp:(g)} and~\ref{hyp:(f)}
  we have
  \begin{align*}
    & \int_0^{t} \int_0^s
      \modulo{f'\left(\tau, X\left(\tau; s, \bar x\right)\right)}
      \modulo{\mathcal E' \left(\tau, s, \bar x\right)
      - \mathcal E'' \left(\tau, s, \bar x\right)} \d \tau \,\d s
    \\
    \le \,
    &
      e^{(G_1 + M) t} \, t\, \norma{m'- m''}
      _{\L\infty([0,t] \times \reali^+; \reali)}
      \int_0^t \int_0^s
      \modulo{f'\left(\tau, X(\tau; s, \bar x)\right)}
      \d \tau\, \d s
    \\
    \le\,
    &
      e^{(G_1 + M) t}  \, t^3\,
      \norma{m'- m''}_{\L\infty([0,t] \times \reali^+; \reali)}
      F_\infty.
  \end{align*}
  Finally using~(\ref{eq:17}) and~\ref{hyp:(g)} we get
  \begin{align*}
    & \int_0^{t} \int_0^s
      \modulo{f'\left(\tau, X\left(\tau; s, \bar x\right)\right)
      - f''\left(\tau, X(\tau; s, \bar x)\right)}
      \modulo{\mathcal E'' \left(\tau, s, \bar x\right)} \d \tau\, \d s
    \\
    \le \,
    &
      e^{(G_1 + M) t}
      \int_0^t \int_0^s
      \modulo{f'\left(\tau, X\left(\tau; s, \bar x\right)\right)
      - f''\left(\tau, X\left(\tau; s, \bar x\right)\right)}
      \d \tau\, \d s
    \\
    \le\,
    &
      e^{(G_1 + M) t}  e^{G_1 t} \int_0^t \int_{\reali^+}
      \modulo{f'\left(\tau, x\right) - f''\left(\tau, x\right)} \d x\,
      \d \tau
    \\
    \le\,
    &
      e^{\left(2 G_1 + M\right) t}
      \norma{f' - f''}_{\L1\left([0,t] \times \reali^+;
      \reali\right)}.
  \end{align*}
  Therefore
  \begin{align*}
    \norma{u'(\cdot, \bar x) - u''(\cdot, \bar x)}
    _{\L1([0,t]; \reali)}
    & \le
      e^{\left(2 G_1 + M\right) t}  \, t\,
      \norma{u_o'}_{\L1(\reali^+; \reali)}
      \norma{m'- m''}
      _{\L\infty([0,t] \times \reali^+; \reali)}
    \\
    & \quad + e^{Mt} \norma{u'_o - u''_o}_{\L1\left(\reali^+; \reali\right)}
    \\
    & \quad + e^{(G_1 + M) t}  \, t^3\, F_\infty \norma{m'- m''}
      _{\L\infty([0,t] \times \reali^+; \reali)}
    \\
    & \quad +
      e^{\left(2 G_1 + M\right) t} \norma{f' - f''}_{\L1\left([0,t] \times \reali^+;
      \reali\right)}
    \\
    & \le e^{(G_1 + M) t} \, t\,
      \left[ e^{G_1 t} \norma{u'_o}_{\L1\left(\reali^+; \reali\right)}
      + t^2 F_\infty\right] \norma{m'- m''}
      _{\L\infty([0,t] \times \reali^+; \reali)}
    \\
    & \quad + e^{Mt} \norma{u'_o - u''_o}_{\L1\left(\reali^+; \reali\right)}
      +
      e^{\left(2 G_1 + M\right) t} \norma{f' - f''}_{\L1\left([0,t] \times \reali^+;
      \reali\right)}
  \end{align*}
  concluding the proof of~(\ref{eq:vertical-stability}).
\end{proof}

\subsection{Proofs Related to Section~\ref{sec:A} -- About the IBVP~\eqref{eq:2}}
\label{sec:ProofsA}

\begin{proofof}{Theorem~\ref{thm:general}}
  Fix $\tmax > 0$, with $\tmax \in \II$, and let $J =
  [0,\tmax]$. Define the constants\marginpar{$[K_1]=n$\\
    $[K_\infty]=\frac{n}{m}$\\}
  \begin{equation}
    \label{eq:constant-C}
    \begin{array}{@{}r@{\;}c@{\;}l@{}}
      K_1
      & >
      & \norma{u_o}_{\L1 (\reali^+; \reali^n)} + B_1
      \\
      K_\infty
      & >
      &
        \max \left\{
        \begin{array}{@{}l@{}}
          \left(1+\frac{n B_L}{\check g}\right) \norma{u_o}_{\L\infty (\reali;\reali^n)} + \frac{B_\infty}{\check g}
          \\
          \left(5 + \frac{G_\infty}{\check g}\right)
          \left(
          \frac{2 B_\infty}{\check g}
          +
          \left(1+\frac{n B_L}{\check g}\right)
          \left(
          \norma{u_o}_{\L\infty (\reali^+; \reali^n)} + \tv (u_o; \reali^+)
          \right)
          \right)
        \end{array}
      \right\}
    \end{array}
  \end{equation}
  and the complete metric space $(\XX^n, \dx)$ where
  \begin{equation}
    \label{eq:space-X}
    \begin{array}{rcl}
      \XX
      & =
      & \left\{
        u \in
        \C0\left(J; \L1 (\reali^+; \reali)\right)
        \colon
        \begin{array}{rcl}
          \displaystyle
          \norma{u}_{\C0(J;\L1(\reali^+; \reali))}
          & \le
          & K_1
          \\
          \displaystyle
          \norma{u}_{\C0(J;\L\infty(\reali^+; \reali))}
          & \le
          & K_\infty
          \\
          \displaystyle
          \sup_{t \in J}
          \tv (u(t, \cdot); \reali^+)
          & \leq
          & K_\infty
        \end{array}
            \right\}
      \\
      \dx (u', u'')
      & =
      & \max\limits_{i\in \{1, \ldots, n\}}
        \norma{u_i'' - u_i'}_{\C0(J;\L1 (\reali^+; \reali))} \,.
    \end{array}
  \end{equation}
  Define the map $\mathcal{T} \colon \XX^n \to \XX^n$, such that, for
  $w = (w_1, \cdots, w_n) \in \XX^n$, $\mathcal T(w) = u$, where
  $u = (u_1, \cdots, u_n)$ solves
  \begin{equation}
    \label{eq:iteratively}
    \left\{
      \begin{array}{l}
        \partial_t u_i + \partial_x (g_i (t,x) \, u_i)
        + m_i (t,x) \, u_i
        =
        f_i (t,x)
        \\
        u_i (0,x) = u_i^o (x)
        \\
        u_i (t,0+) = b_i (t)
      \end{array}
    \right.
    \quad
    i=1, \ldots, n
  \end{equation}
  where
  \begin{equation}
    \label{eq:19}
    \begin{array}{rcl}
      m_i (t,x)
      & =
      & - \left(\alpha_i[w (t)] (x)\right)_i - \left(\gamma_i (t,x)\right)_i
      \\
      f_i (t,x)
      & =
      & \sum\limits_{j\neq i}
        \left(\left(\alpha_i[w (t)] (x)\right)_j + \left(\gamma_i (t,x)\right)_j\right) w_j (t,x)
      \\
      b_i (t)
      & =
      & \beta_i \left(t, u_1 (t, \bar x_1-), \ldots, u_n (t, \bar x_n-)\right) \,.
    \end{array}
  \end{equation}
  Remark that in the last line above an essential role is going to be
  played by the assumption $\partial_{u_j} \beta_i (t,u) = 0$ for all
  $j \geq i$.

  \paragraph{The map $\mathcal{T}$ is well defined.} (i.e.
  $\mathcal T(w) \in \XX^n$ for every $w \in \XX^n$). Aiming at the
  use of Lemma~\ref{lem:stability}, we verify that the
  assumptions~\textup{\ref{hyp:(g)}}, \textup{\ref{hyp:(m)}},
  \textup{\ref{hyp:(f)}}, and~\textup{\ref{hyp:(b)}} therein hold.

  \subparagraph{\textup{\ref{hyp:(g)}} holds.} It is immediate
  by~\ref{hyp:g}.

  \subparagraph{\textup{\ref{hyp:(m)}} holds.} The continuity of
  $x \to m (t,x)$ follows from~\ref{hyp:A_i}
  and~\ref{hyp:gamma_i}. Observe that the map
  $t \to \alpha_i[w (t)] (x)$ is continuous, indeed:
  \begin{flalign*}
    \norma{\alpha_i[w (t_2)] (x) - \alpha_i[w (t_1)] (x)} \leq &
    \norma{\alpha_i[w (t_2)] - \alpha_i[w (t_1)]}_{\C0 (\reali^+;
      \reali^n)}
    \\
    \leq & \norma{\alpha_i[w (t_2) - w (t_1)]}_{\C0 (\reali^+;
      \reali^n)} & \mbox{[By linearity}]
    \\
    \leq & A_L \, \norma{w (t_2) - w (t_1)}_{\L1 (\reali^+; \reali)} &
    \mbox{[By~\eqref{eq:hyp-A1}]}
  \end{flalign*}
  and the fact that $w \in \C0 (J; \L1 (\reali^+; \reali))$ allows to
  conclude.

  Fix $(t_1,x_1), \, (t_2,x_2)$ in $J\times\reali^+$ and compute, for
  $i=1, \ldots, n$,
  \begin{flalign*}
    & \modulo{m_i (t,x_2) - m_i (t,x_1)}
    \\
    \leq & \modulo{\left(\alpha_i[w (t)]\right)_i (x_2) -
      \left(\alpha_i[w (t)]\right)_i (x_2)} + \modulo{\left(\gamma_i
        (t,x_2)\right)_i - \left(\gamma_i (t,x_1)\right)_i} &
    \mbox{[By~\eqref{eq:19}]}
    \\
    \leq & \norma{\alpha_i[w (t)] (x_2) - \alpha_i[w (t)] (x_2)} +
    \norma{\gamma_i (t,x_2) - \gamma_i (t,x_1)}
    \\
    \leq & A_2 \, \modulo{x_2 - x_1} + C_L\, \modulo{x_2 - x_1} \,. &
    \mbox{[By~\eqref{eq:hyp-A4} and~\eqref{eq:15}}
  \end{flalign*}
  Moreover,
  \begin{flalign*}
    \norma{m_i}_{\L\infty (J\times\reali^+;\reali)} \leq &
    \supess_{t\in J} \norma{\left(\alpha_i[w
        (t)]\right)_i}_{\C0(\reali^+; \reali)} +
    \norma{\gamma_i}_{\L\infty (J\times \reali^+; \reali^n)} &
    \mbox{[By~\eqref{eq:19}]}
    \\
    \leq & A_L \, \norma{w}_{\C0 (J;\L1 (\reali^+; \reali^n))} +
    C_\infty & \mbox{[By~\eqref{eq:hyp-A1} and~\eqref{eq:20}]}
    \\
    \leq & A_L \, K_1 + C_\infty & \mbox{[By~\eqref{eq:space-X}]}
  \end{flalign*}
  so that $m_i$ is a Caratheodory function. Finally,
  \begin{flalign*}
    \sup_{t \in J} \tv (m_i (t, \cdot); \reali^+) \leq & \sup_{t \in
      J} \tv (\alpha_i[w (t)]; \reali^+) + \sup_{t \in J} \tv
    (\gamma_i (t, \cdot); \reali^+) & \mbox{[By~\eqref{eq:19}]}
    \\
    \leq & A_L \, \norma{w}_{\C0 (J; \L1 (\reali^+; \reali))} +
    C_\infty & \mbox{[By~\eqref{eq:hyp-A1} and~\eqref{eq:20}]}
    \\
    \leq & A_L \, K_1 + C_\infty & \mbox{[By~\eqref{eq:space-X}]}
  \end{flalign*}
  completing the proof of~\ref{hyp:(m)} with
  \begin{equation}
    \label{eq:34}
    M = A_L \, K_1 + C_\infty\,.
  \end{equation}

  \subparagraph{\textup{\ref{hyp:(f)}} holds.}  Compute the terms in
  the right hand side of
  \begin{flalign*}
    \norma{f_i (t_2) - f_i (t_1)}_{\L1 (\reali^+; \reali)} \leq &
    \sum_{j\neq i} \norma{ \left(\alpha_i[w (t_2)] \right)_j w_j (t_2)
      - \left(\alpha_i[w (t_1)] \right)_j w_j (t_1)}_{\L1 (\reali^+;
      \reali)}
    \\
    & + \sum_{j\neq i} \norma{ \left(\gamma_i (t_2)\right)_j w_j (t_2)
      - \left(\gamma_i (t_1)\right)_j w_j (t_1)}_{\L1 (\reali^+;
      \reali)}
  \end{flalign*}
  separately, obtaining
  \begin{flalign*}
    & \norma{ \left(\alpha_i[w (t_2)] \right)_j w_j (t_2) -
      \left(\alpha_i[w (t_1)] \right)_j w_j (t_1)}_{\L1 (\reali^+;
      \reali)} & \mbox{[Use the linearity of $\alpha_i$]}
    \\
    \leq & \norma{\alpha_i[w (t_2)]}_{\C0 (\reali^+; \reali^n)}
    \norma{w (t_2) - w (t_1)}_{\L1 (\reali^+; \reali^n)} &
    \mbox{[Use~\eqref{eq:hyp-A1}]}
    \\
    & + \norma{\alpha_i[w (t_2) - w (t_1)]}_{\C0 (\reali^+; \reali^n)}
    \norma{w (t_1)}_{\L1 (\reali^+; \reali^n)} &
    \mbox{[Use~\eqref{eq:hyp-A1}]}
    \\
    \leq & A_L \, \norma{w (t_2)}_{\L1 (\reali^+; \reali)} \norma{w
      (t_2) - w (t_1)}_{\L1 (\reali^+; \reali^n)} &
    \mbox{[Use~\eqref{eq:space-X}]}
    \\
    & + A_L \, \norma{w (t_2) - w (t_1)}_{\L1 (\reali^+; \reali^n)} \,
    \norma{w (t_1)}_{\L1 (\reali^+; \reali)} &
    \mbox{[Use~\eqref{eq:space-X}]}
    \\
    \leq & 2\, A_L \, K_1 \, \norma{w (t_2) - w (t_1)}_{\L1 (\reali^+;
      \reali^n)} \,.
  \end{flalign*}
  and similarly
  \begin{flalign*}
    &\norma{ \left(\gamma_i (t_2)\right)_j w_j (t_2) - \left(\gamma_i
        (t_1)\right)_j w_j (t_1)}_{\L1 (\reali^+; \reali)}
    \\
    \leq & \norma{\gamma_i(t_2)}_{\L\infty (\reali^+; \reali^n)}
    \norma{w (t_2) - w (t_1)}_{\L1 (\reali^+; \reali^n)} &
    \mbox{[Use~\eqref{eq:20}]}
    \\
    & + \norma{w (t_1)}_{\L\infty (\reali^+; \reali^n)}
    \norma{\gamma_i(t_2) - \gamma_i (t_1)}_{\L1 (\reali^+; \reali^n)}
    & \mbox{[Use~\eqref{eq:space-X}]}
    \\
    \leq & C_\infty \norma{w (t_2) - w (t_1)}_{\L1 (\reali^+;
      \reali^n)} + K_\infty \norma{\gamma_i(t_2) - \gamma_i
      (t_1)}_{\L1 (\reali^+; \reali^n)}
  \end{flalign*}
  which show that $f_i \in \C0 (J; \L1 (\reali^+; \reali))$,
  by~\ref{hyp:gamma_i} and~\eqref{eq:space-X}.

  We prove now the $\L1$ and $\L\infty$ bounds on $f$:
  \begin{flalign*}
    & \norma{f_i}_{\L\infty(J; \L1 (\reali^+;\reali))}
    \\
    \leq & \sum_{j\neq i} \left( \norma{ \left(\alpha_i[w (t)]
        \right)_j w_j (t)}_{\L1 (\reali^+; \reali)} + \norma{
        \left(\gamma_i (t)\right)_j w_j (t_2)}_{\L1 (\reali^+;
        \reali)} \right) & \mbox{[By~\eqref{eq:19}]}
    \\
    \leq & (A_L \, {K_1}^2 + C_\infty \, K_1)n &
    \mbox{[By~\eqref{eq:hyp-A1}, \eqref{eq:20}
      and~\eqref{eq:space-X}]}
  \end{flalign*}
  proving the $\L\infty$ bound on $f_i$ with
  \begin{equation}
    \label{eq:33}
    F_1 = (A_L \, {K_1}^2 + C_\infty \, K_1)n \,.
  \end{equation}
  The $\L\infty$ bound is proved similarly:
  \begin{flalign*}
    & \norma{f_i}_{\L\infty(J \times \reali^+;\reali))}
    \\
    \leq & \sum_{j\neq i} \left( \norma{ \left(\alpha_i[w] \right)_j
        w_j }_{\L\infty (J\times\reali^+; \reali)} + \norma{
        (\gamma_i)_j w_j}_{\L\infty (J\times\reali^+; \reali)} \right)
    & \mbox{[By~\eqref{eq:19}]}
    \\
    \leq & (A_L \, K_1 \, K_\infty + C_\infty \, K_\infty)n &
    \mbox{[By~\eqref{eq:hyp-A1}, \eqref{eq:20}
      and~\eqref{eq:space-X}]}
  \end{flalign*}
  Moreover,
  \begin{flalign*}
    & \tv (f_i (t, \cdot); \reali^+)
    \\
    \leq & \sum_{j\neq i} \tv\left( \left(\alpha_i[w (t)]
        (\cdot)\right)_j \, w_j (t,\cdot) \right) + \sum_{j\neq i}
    \tv\left( \left(\gamma_i (t,\cdot)\right)_j \, w_j (t, \cdot)
    \right) & \mbox{[By~\eqref{eq:19}}
    \\
    \leq & \sum_{j\neq i} \tv\left( \left(\alpha_i[w (t)]
        (\cdot)\right)_j; \reali^+\right) \norma{w_j (t)}_{\L\infty
      (\reali^+; \reali)} & \mbox{[Use~\eqref{eq:hyp-A1}
      and~\eqref{eq:space-X}]}
    \\
    &+ \sum_{j\neq i} \norma{\left(\alpha_i[w (t)]
        (\cdot)\right)_j}_{\L\infty (\reali^+; \reali)} \tv\left(w_j
      (t,\cdot); \reali^+ \right) & \mbox{[Use~\eqref{eq:hyp-A1}
      and~\eqref{eq:space-X}]}
    \\
    & + \sum_{j\neq i} \tv\left((\gamma_i(t))_j ; \reali^+ \right)
    \norma{w_j}_{\L\infty (J\times\reali^+; \reali)} &
    \mbox{[Use~\eqref{eq:20} and~\eqref{eq:space-X}]}
    \\
    & + \sum_{j\neq i} \norma{(\gamma_j (t)_i)}_{\L\infty (\reali^+;
      \reali)} \tv\left(w_j (t); \reali^+\right) &
    \mbox{[Use~\eqref{eq:20} and~\eqref{eq:space-X}]}
    \\
    \leq & 2 \, n\, A_L \, K_\infty \, K_1 + 2 \, n\, C_\infty
    K_\infty
  \end{flalign*}
  completing the proof that~\ref{hyp:(f)} holds with
  \begin{equation}
    \label{eq:32}
    F_\infty = 2nK_\infty (C_\infty + A_L\, K_1) \,.
  \end{equation}

  \subparagraph{\textup{\ref{hyp:(b)}} holds for $i=1$.} Note that, in
  this case, $\beta_1 (t,u)$ is independent of $u$, so that
  using~\eqref{eq:barB-4},
  $\tv(b_1; J) = \tv (\beta_1 (\cdot); J) \leq B_\infty$.

  \subparagraph{$\mathcal{T}w$ is Lipschitz continuous in time with
    respect to the $\L1$ norm.} Simply apply~\ref{it:S:6}, observing
  that~\ref{hyp:(f)}, \ref{hyp:(g)}, \ref{hyp:(m)} and~\ref{hyp:(b)}
  were proved above exhibiting bounds that hold uniformly on
  $\mathcal{X}^n$, once the norm of the initial datum $u_o$ and the
  constants in~\ref{hyp:g}--\ref{hyp:init-boundary} are fixed.

  \subparagraph{$\left(\mathcal{T} (w)\right)_1 \in \mathcal{X}$.}  By
  Lemma~\ref{lem:stability}, $u_1$ is well defined as solution to the
  Initial Boundary Value Problem~\eqref{eq:iteratively} with $i=1$.
  By~\ref{it:S:2}, \ref{hyp:B_i} and~\eqref{eq:barB-2}, we have that
  \begin{align*}
    \norma{u_1}_{\C0(J;\L1(\reali^+; \reali))}
    & \le
      \left[\norma{u_1^o}_{\L1(\reali^+; \reali)}
      + \frac{1}{\check g} \norma{\beta_1}_{\L1(J; \reali)}
      + (A_1 {K_1}^2 + C_\infty \, K_1) n \, \tmax
      \right]
      e^{(A_L K_1 +C_\infty)\tmax}
    \\
    & \le
      \left[
      \norma{u_1^o}_{\L1(\reali^+; \reali)}
      + \frac{B_1}{\check g}
      + (A_1 {K_1}^2 + C_\infty \, K_1) n \, \tmax
      \right]
      e^{(A_L K_1 +C_\infty)\tmax}
    \\
    & \leq
      K_1 \quad \mbox{ provided $\tmax$ is small, since by~\eqref{eq:constant-C} }
      K_1 > \norma{u_1^o}_{\L1(\reali^+; \reali)}
      + \frac{B_1}{\check g}\,.
  \end{align*}
  To bound the $\L\infty$ norm, Use~\ref{it:S:2} and~\eqref{eq:barB-3}
  to obtain
  \begin{eqnarray*}
    \norma{u_1 (\tau)}_{\L\infty ([0, \tmax] \times \reali^+; \reali)}
    & \leq
    & \left(
      \norma{u_1^o}_{\L\infty (\reali^+; \reali)}
      {+}
      \frac{B_\infty}{\check g}
      {+}
      2 \, n \, K_\infty (C_\infty + A_L\, K_1)\, \tmax
      \right)
      e^{(G_1 + A_L K_1 + C_\infty) \tmax}
    \\
    & <
    & K_\infty
      \quad \mbox{ provided $\tmax$ is small, since }
      K_\infty
      >
      \norma{u_1^o}_{\L\infty (\reali^+; \reali)} + \frac{B_\infty}{\check g} \,.
  \end{eqnarray*}
  Using~\ref{it:S:5}, we have
  $u_1 \in \C0\left(J; \L1(\reali^+; \reali)\right)$.  By~\ref{it:S:3}
  and~\ref{hyp:B_i}, for $t \in J$, we have
  \begin{eqnarray*}
    \tv (u_1(t); \reali^+)
    & \leq
    & \mathcal{H}(t)
      \left( \!
      F_\infty t
      +
      \frac{\norma{\beta_1}_{\L\infty ([0,t]; \reali)}
      +
      \tv (\beta_1; \reali^+)
      }{\check g}
      +
      \norma{u_1^o}_{\L\infty (\reali^+; \reali)}
      +
      \tv (u^o_1; \reali^+)
      \! \right)
    \\
    & \leq
    & \mathcal{H}(\tmax)
      \left(
      2n\, K_\infty (C_\infty+A_L \, K_1)\, \tmax
      +
      \frac{2}{\check g} B_\infty
      +
      \norma{u^o_1}_{\L\infty (\reali^+; \reali)}
      +
      \tv (u^o_1; \reali^+)
      \right)
    \\
    & \leq
    & K_\infty
      \quad \mbox{ provided $\tmax$ is small, since by~\eqref{eq:constant-C} }\\
    &
    & K_\infty
      >
      \left(5+\frac{G_\infty}{\check g}\right)
      \left(\frac{2B_\infty}{\check g} + \norma{u_o}_{\L\infty (\reali^+; \reali^n)} + \tv (u_o; \reali^+)\right) \,,
  \end{eqnarray*}
  completing the proof that $u_1 = (\mathcal{T}w)_1\in \mathcal{X}$.
  Remark for later use that, by~\ref{it:S:4},
  \begin{equation}
    \label{eq:bv-vertical}
    u_1 (\cdot, \bar x_1) \in \BV(J; \reali) \,.
  \end{equation}

  \subparagraph{\textup{\ref{hyp:(b)}} holds for $i>1$.}  Fix now an
  index $i > 1$, assume that $u_1, \ldots, u_{i-1} \in \XX$ and
  consider the Initial Boundary Value
  Problem~\eqref{eq:iteratively}--\eqref{eq:19}.  By~\ref{hyp:B_i},
  the function $\beta_i$ depends only on $t$ and on
  $u_1, \cdots, u_{i-1}$. Moreover, by~\ref{hyp:B_i}
  and~(\ref{eq:bv-vertical}), which hold for every $u_j$ with $j < i$,
  the map
  \begin{equation*}
    t \mapsto b_i (t) =  \beta_i \left(t, u_1(t, \bar x_1), \cdots,
      u_{i - 1}(t, \bar x_{i-1})\right)
  \end{equation*}
  is of bounded variation and hence satisfies~\ref{hyp:(b)}.

  \subparagraph{$\left(\mathcal{T} (w)\right)_i \in \mathcal{X}$ for
    $i > 1$.}  Lemma~\ref{lem:stability} implies that there exists a
  solution $u_i$ to~\eqref{eq:iteratively}.  By~\ref{it:S:2}
  and~\eqref{eq:barB-1}~\eqref{eq:barB-2}, we have that
  \begin{eqnarray*}
    &
    & \norma{u_i (t)}_{\C0(J;\L1(\reali^+; \reali))}
    \\
    & \le
    &  \left[\norma{u_i^o}_{\L1(\reali^+; \reali)}
      + \norma{b_i}_{\L1(J; \reali)}
      + F_1 \, \tmax
      \right]
      e^{M \tmax}
    \\
    & \le
    &  \left[
      \norma{u_i^o}_{\L1(\reali^+; \reali)}
      + B_1 + B_L \sum_{j = 1}^{i-1}
      \int_0^{\tmax}\modulo{u_j(s, \bar x_j)} \d s
      + n (A_L \, {K_1}^2 + C_\infty K_1) \tmax
      \right] e^{(A_L K_1 + C_\infty)\tmax}
    \\
    & \le
    &  \left[
      \norma{u_i^o}_{\L1(\reali^+; \reali)}
      + B_1 + (i-1) B_L \, K_\infty \tmax
      + n (A_L \, {K_1}^2 + C_\infty K_1) \tmax
      \right] e^{(A_L K_1 + C_\infty)\tmax}
    \\
    & \leq
    & K_1 \quad
      \mbox{ provided $\tmax$ is small, since by~\eqref{eq:constant-C} }
      K_1 > \norma{u_i^o}_{\L1(\reali^+; \reali)}
      + B_1 \,.
  \end{eqnarray*}
  Using~\ref{it:S:5},
  $u_i \in \C0\left(J; \L1(\reali^+; \reali)\right)$.  To bound the
  $\L\infty$ norm, Use~\ref{it:S:2}, \eqref{eq:12}, \eqref{eq:barB-1},
  and~\eqref{eq:barB-3} to obtain
  \begin{eqnarray*}
    &
    & \norma{u_i}_{\L\infty (J\times\reali^+; \reali)}
    \\
    & \leq
    & \left(
      \norma{u_i^o}_{\L\infty (\reali^+; \reali)}
      +
      \frac{\norma{b_i}_{\L\infty(J; \reali)}}{\check g}
      +
      F_\infty \, \tmax
      \right)
      \,  e^{(G_1 + M) \tmax}
    \\
    & \le
    &  \left(\norma{u_i^o}_{\L\infty (\reali^+; \reali)}
      + \frac{B_\infty}{\check g}
      + \frac{B_L}{\check g}
      \sum_{j=1}^{i-1}\modulo{u_j (t, \bar x_j)}
      +
      2\, n \, K_\infty (C_\infty+A_L \, K_1) \tmax
      \right) e^{(G_1 + A_L K_1+C_\infty) \tmax}
    \\
    & \le
    &  \biggl(\left(1 + \frac{n B_L}{\check g}
      e^{(G_1 + A_L K_1+C_\infty) \tmax} \right)
      \norma{u_o}_{\L\infty(\reali^+; \reali^n)}
      \biggr.
    \\
    &
    &\qquad
      \biggl.
      + \frac{B_\infty}{\check g}
      + \frac{2 n B_LK_\infty}{\check g}
      (C_\infty + A_L \, K_1) \tmax e^{(G_1+A_L K_1+C_\infty) \tmax}
      +       2\, n \, K_\infty (C_\infty+A_L \, K_1) \tmax
      \biggr)
    \\
    & &\quad
        \times e^{(G_1 + A_L K_1+C_\infty) \tmax}
    \\
    & <
    & K_\infty \quad
      \mbox{provided $\tmax$ is small, since by~\eqref{eq:constant-C} }
      K_\infty > \left(1+\frac{n B_L}{\check g}\right) \norma{u_o}_{\L\infty (\reali;\reali^n)} + \frac{B_\infty}{\check g} \,.
  \end{eqnarray*}

  We pass now to estimate the total variation. By~\ref{it:S:3}
  and~\ref{hyp:B_i}, for $t \in J$, we have
  \begin{eqnarray*}
    &
    & \tv (u_i(t); \reali^+)
    \\
    & \leq
    & \mathcal{H}(t)
      \left(
      F_\infty t
      +
      \frac{\norma{b_i}_{\L\infty ([0,t]; \reali)}
      +
      \tv (b_i; J)
      }{\check g}
      +
      \norma{u_i^o}_{\L\infty (\reali^+; \reali)}
      +
      \tv (u^o_i; \reali^+)
      \right)
    \\
    & \leq
    & \mathcal{H}(\tmax) \!\!
      \left( \!\!
      2 n K_\infty (C_\infty {+} A_L \, K_1) \tmax
      {+}
      \frac{1}{\check g} \!
      \left(\!
      2B_\infty
      {+}
      B_L
      \sum_{j=1}^{i-1}
      \left(
      \norma{u_j (\cdot, \bar x_j)}_{\L\infty (J; \reali)}
      {+}
      \tv (u_j (\cdot, \bar x_j);J)
      \!\right)
      \!\right)
      \right.
    \\
    &
    & \qquad\quad
      +
      \left.
      \norma{u_i^o}_{\L\infty (\reali^+; \reali)}
      +
      \tv (u^o_i; \reali^+)^{\phantom{\Big|}}
      \right)
      \qquad\qquad\qquad\qquad\;\,\mbox{[By~\eqref{eq:barB-1}, \eqref{eq:barB-3} and~\eqref{eq:barB-4}]}
    \\
    & \leq
    & \mathcal{H}(\tmax)
      \left(
      2 \, n \, K_\infty (C_\infty + A_L \, K_1)\, \tmax
      +
      \frac{2B_\infty}{\check g}
      \right.
      \qquad\qquad\qquad\qquad
      \mbox{[Use~\eqref{eq:12}, \eqref{eq:17}, \ref{hyp:(f)}, \eqref{eq:tv-E-vertical}]}
    \\
    &
    & \qquad\quad
      \left.
      +
      \frac{n \, B_L}{\check g}
      \left(
      \norma{u^o}_{\L\infty(\reali^+; \reali^n)}
      +
      2 \, n \, K_\infty (C_\infty + A_L \, K_1) \tmax
      \right)
      e^{(G_1 + A_L K_1 + C_\infty)\tmax}
      \right.
    \\
    &
    &\qquad\quad
      +
      \frac{n \, B_L}{\check g}
      \left[
      \tv (u_o;\reali^+)
      +
      2 (G_1 + A_L \, K_1 + C_\infty) \norma{u_o}_{\L\infty(\reali^+; \reali^n)} \tmax\right]
      e^{(G_1 + A_L \, K_1 + C_\infty)\, \tmax}
    \\
    &
    & \qquad\quad
      +
      \frac{8 \, n^2 \, B_L}{\check g}
      [ 1 + (G_1 + A_L \, K_1 + C_\infty) \tmax ] \,
      K_\infty (A_L \, K_1 + C_\infty)  \, \tmax \,
      e^{(G_1 + A_L K_1 + C_\infty)\, \tmax}
    \\
    &
    & \qquad\quad+
      \left.
      \norma{u_o}_{\L\infty (\reali^+; \reali^n)}
      +
      \tv (u_o; \reali^+)^{\phantom{|}}
      \right)
    \\
    & <
    & K_\infty \qquad
      \mbox{provided $\tmax$ is small, since by~\eqref{eq:constant-C} }\\
    &
    &  K_\infty > \left(5 + \frac{G_\infty}{\check g}\right)
      \left(
      \frac{2 B_\infty}{\check g}
      +
      \frac{n B_L}{\check g}
      \left(
      \norma{u_o}_{\L\infty (\reali^+; \reali^n)} + \tv (u_o; \reali^+)
      \right)
      \right) \,.
  \end{eqnarray*}

  This concludes the proof of the well posedness of $\mathcal T$.

  \paragraph{The map $\mathcal{T}$ is a contraction.}

  Fix $w, \bar w \in \XX$. For later use, we prepare some estimates.
  By~\ref{hyp:A_i} and~(\ref{eq:space-X}), for every
  $i \in \left\{1, \cdots, n\right\}$, we deduce that
  \begin{equation}
    \label{eq:alpha-i_infty}
    \norma{
      (\alpha_i[w])_i - (\alpha_i[\bar w])_i
    }_{\L\infty(J \times \reali^+; \reali)}
    \le A_L d_{\XX^n} (w, \bar w) \,.
  \end{equation}
  Moreover, by~\ref{hyp:A_i} and~(\ref{eq:space-X}), for every
  $i \in \{1, \cdots, n\}$ and
  $j \in \{1, \cdots, n\} \setminus \{i\}$, we obtain
  \begin{equation}
    \label{eq:alpha-i_1}
    \begin{split}
      & \norma{(\alpha_i[w])_j \, w_j - (\alpha_i[\bar w])_j \, \bar
        w_j} _{\L1(J \times \reali^+;\reali)}
      \\
      \le & \norma{(\alpha_i[w])_j \, w_j - (\alpha_i[\bar w])_j \,
        w_j} _{\L1(J \times \reali^+;\reali)} + \norma{(\alpha_i[\bar
        w])_j (w_j - \bar w_j)} _{\L1(J \times \reali^+;\reali)}
      \\
      \le & \norma{(\alpha_i [w - \bar w])_j}_{\L\infty(J \times
        \reali^+; \reali)} \, d_{\XX^n} (w, 0) \, \tmax +
      \norma{(\alpha_i[\bar w])_j}_{\L\infty(J \times \reali^+;
        \reali)} \, \tmax\, d_{\XX^n} (w, \bar w)
      \\
      \le & 2 \, A_L \, K_1 \, \tmax\; d_{\XX^n} (w, \bar w) \,.
    \end{split}
  \end{equation}
  Finally by~\ref{hyp:gamma_i}, for every $i \in \{1, \cdots, n\}$ and
  $j \in \{1, \cdots, n\} \setminus \{i\}$, we also have
  \begin{equation}
    \label{eq:gamma-i_1}
    \norma{
      (\gamma_i)_j \, w_j - (\gamma_i)_j \, \bar w_j }_{\L1(J
      \times \reali^+;\reali)}
    \le
    C_\infty \, \tmax \, d_{\mathcal{X}^n} (w, \bar w) \,.
  \end{equation}

  For $i = 1$ and $t \in J$, Lemma~\ref{lem:stability2} implies that
  \begin{flalign*}
    & \norma{(\mathcal{T} w)_1(t, \cdot) - (\mathcal{T} \bar w)_1(t,
      \cdot)} _{\L1(\reali^+; \reali)} & \!\!\!
    \mbox{[Use~\eqref{eq:stability-linear} and~\eqref{eq:19}]}
    \\
    \le & 2 e^{2 (G_1 + A_L K_1 + C_\infty) \tmax}\,
    \sum_{j=2}^n\norma{ (\alpha_1[w])_j \, w_j - (\alpha_1[\bar w])_j
      \, \bar w_j }_{\L1(J \times \reali^+;\reali)}
    \\
    & + 2 e^{2 (G_1 + A_L K_1 + C_\infty) \tmax}\, \sum_{j=2}^n\norma{
      (\gamma_1)_j \, w_j - (\gamma_1)_j \, \bar w_j }_{\L1(J \times
      \reali^+;\reali)}
    \\
    & + e^{(2G_1 + A_L K_1 + C_\infty) \tmax} \left[
      \norma{u^o_1}_{\L1(\reali^+; \reali)} + 2 \, n K_1 (A_L \, K_1 +
      C_\infty) \tmax + B_1\right]
    \\
    & \qquad \times \tmax\, \norma{(\alpha_1[w])_1 - (\alpha_1[\bar
      w])_1}_{\L\infty(J \times \reali^+; \reali)}.
  \end{flalign*}
  Therefore, using~(\ref{eq:alpha-i_infty}), (\ref{eq:alpha-i_1}),
  and~(\ref{eq:gamma-i_1}), we get that
  \begin{flalign*}
    & \norma{(\mathcal{T} w)_1(t, \cdot) - (\mathcal{T} \bar w)_1(t,
      \cdot)} _{\L1(\reali^+; \reali)}
    \\
    \le & 4 \, n \, e^{2 (G_1 + A_L K_1 + C_\infty) \tmax}\, A_L \,
    K_1 \, \tmax\, d_{\XX^n} (w, \bar w) + 2 \, n \, e^{2 (G_1 + A_L
      K_1 + C_\infty) \tmax}\, C_\infty \, \tmax \, d_{\mathcal{X}^n}
    (w, \bar w)
    \\
    & + e^{(2G_1 + A_L K_1 + C_\infty) \tmax} \left[
      \norma{u^o_1}_{\L1(\reali^+; \reali)} + 2 \, n K_1 (A_L \, K_1 +
      C_\infty) \tmax + B_1\right] A_L \, d_{\mathcal{X}^n} (w, \bar
    w)
  \end{flalign*}
  and so, choosing $\tmax$ sufficiently small, we obtain that
  \begin{equation}
    \label{eq:contraction-1}
    \norma{(\mathcal{T} w)_1 - (\mathcal{T} \bar w)_1}_{\C0(J;\L1(\reali^+; \reali))}
    \le \frac{1}{2n} d_{\XX^n} (w, \bar w).
  \end{equation}

  For $i > 1$ and $t \in J$, Lemma~\ref{lem:stability2} implies that
  \begin{align*}
    & \norma{(\mathcal{T} w)_i(t, \cdot) - (\mathcal{T} \bar w)_i(t, \cdot)}
      _{\L1(\reali^+; \reali)}
    \\
    \le
    & 2 e^{2 (G_1 + A_L K_1 + C_\infty) \tmax}\,
      \sum_{j =1, j \ne i}^n\norma{ (\alpha_i[w])_j \, w_j - (\alpha_i[\bar w])_j
      \, \bar w_j }_{\L1(J \times \reali^+;\reali)}
    \\
    & + 2 e^{2 (G_1 + A_L K_1 + C_\infty) \tmax}\, \sum_{j=2, j\ne i}^n\norma{
      (\gamma_i)_j \, w_j - (\gamma_i)_j \, \bar w_j }_{\L1(J
      \times \reali^+;\reali)}
    \\
    & + e^{2 (G_1 + A_L K_1 + C_\infty) \tmax}
      \norma{b_i - \bar b_i}_{\L1(J; \reali)}
    \\
    & +
      e^{(2G_1 + A_L K_1 + C_\infty) \tmax}
      \left[\norma{u^o_1}_{\L1(\reali^+; \reali)}
      + 2 n K_1 (A_L K_1 + C_\infty) \tmax + \norma{\bar b_i}
      _{\L1(J; \reali)}\right]
    \\
    & \quad \times
      \tmax\,
      \norma{(\alpha_i[w])_i -
      (\alpha_i[\bar w])_i}_{\L\infty(J \times \reali^+; \reali)}
  \end{align*}
  where
  \begin{equation}
    \label{eq:bar-b_i}
    \!\!\!
    b_i(t) = \beta_i\left(t, u_1(t, \bar x_1),
      \cdots, u_{i-1}(t, \bar x_{i-1})\right)
    \ \mbox{ and }\
    \bar b_i(t) = \beta_i\left(t, \bar u_1(t, \bar x_1),
      \cdots, \bar u_{i-1}(t, \bar x_{i-1})\right)
  \end{equation}
  are the boundary terms respectively for $w$ and $\bar w$.
  We thus have that
  \begin{flalign*}
    & \norma{b_i - \bar b_i} _{\L1(J; \reali)}
    \\
    \le & B_L \sum_{j = 1}^{i-1} \norma{u_j(\cdot, \bar x_j) - \bar
      u_j (\cdot, \bar x_j)}_{\L1(J; \reali)} &
    \mbox{[By~(\ref{eq:barB-1}), (\ref{eq:bar-b_i})]}
    \\
    \le & B_L \sum_{j=1}^{i-1} e^{(G_1 + A_L K_1 + C_\infty) \tmax}
    \tmax\! \left[ \!e^{G_1 \tmax} \norma{u_j^o}_{\L1\left(\reali^+;
          \reali\right)} {+} \tmax\,^2 2 n K_\infty (A_L K_1 +
      C_\infty)\!\right] &
    \mbox{[By~(\ref{eq:vertical-stability}), (\ref{eq:19})]}
    \\
    & \quad \times \norma{\left(\alpha_j[w]\right)_j -
      \left(\alpha_j[\bar w]\right)_j} _{\L\infty(J \times \reali^+;
      \reali)}
    \\
    & + e^{\left(2 G_1 + A_L K_1 + C_\infty\right) \tmax} \sum_{h=1, h
      \ne j}^{n}\norma{\left(\alpha_j[w]\right)_h w_h -
      \left(\alpha_j[\bar w]\right)_h \bar w_h}_{\L1(J \times
      \reali^+; \reali)}
    \\
    & + e^{\left(2 G_1 + A_L K_1 + C_\infty\right) \tmax} \sum_{h=1, h
      \ne j}^{n}\norma{(\gamma_j)_h \,w_h - (\gamma_j)_h \, \bar
      w_h}_{\L1(J \times \reali^+;\reali)}
    \\
    \le & B_L n e^{(G_1 + A_L K_1 + C_\infty) \tmax} \tmax \! \left[
      e^{G_1 \tmax} \norma{u_o}_{\L1(\reali^+; \reali^n)} + {\tmax}^2
      2 n K_\infty (A_L K_1 + C_\infty)\!\right]
    \\
    & \qquad \times A_L d_{\XX^n}(w, \bar w) &
    \mbox{[By~(\ref{eq:alpha-i_infty})]}
    \\
    & + e^{(2 G_1 + A_L K_1 + C_\infty) \tmax} 2 nA_L K_1 \tmax
    d_{\XX^n}(w, \bar w) & \mbox{[By~(\ref{eq:alpha-i_1})]}
    \\
    & + e^{(2 G_1 + A_L K_1 + C_\infty) \tmax} 2 n C_\infty \tmax
    d_{\XX^n}(w, \bar w).  & \mbox{[By~(\ref{eq:gamma-i_1})]}
  \end{flalign*}
  Moreover
  \begin{flalign*}
    & \norma{\bar b_i} _{\L1(J; \reali)}
    \\
    \le & B_L \sum_{j = 1}^{i-1} \norma{u_j(\cdot, \bar x_j) }_{\L1(J;
      \reali)} + B_1 & \mbox{[By~\ref{eq:barB-1}, ~\ref{eq:barB-2}
      and~(\ref{eq:bar-b_i})]} \\ \le & n B_L e^{(2 G_1 + A_L K_1 +
      C_\infty) \tmax} \, \tmax \, \norma{u_o}_{\L\infty(\reali^+;
      \reali^n)} & \mbox{[By~\ref{eq:12}, ~(\ref{eq:17}),
      and~\ref{hyp:(g)}]}
    \\
    & + 2 n^2 B_L e^{(G_1 + A_L K_1 + C_\infty) \tmax} K_\infty(A_L
    K_1 + C_\infty) \, \tmax\,^2.  & \mbox{[By~\ref{eq:12},
      ~(\ref{eq:17}), and~\ref{hyp:(f)}]}
  \end{flalign*}

  Finally, using again~(\ref{eq:alpha-i_infty}), (\ref{eq:alpha-i_1}),
  and~(\ref{eq:gamma-i_1}), we obtain
  \begin{align*}
    & \norma{(\mathcal{T}w)_i(t, \cdot) - (\mathcal{T} \bar w)_i(t, \cdot)}
      _{\L1(\reali^+; \reali)}
    \\
    \le
    & 4 \, n \, A_L \, K_1 \; \tmax \, d_{\XX^n} (w, \bar w)
      \; e^{2 (G_1 + A_L K_1 + C_\infty) \tmax}
    \\
    & + 2 \, n \, C_\infty \; \tmax \, d_{\XX^n} (w, \bar w) \,
      e^{2 (G_1 + A_L K_1 + C_\infty) \tmax}
    \\
    & + n \, A_L \, B_L \, \left[e^{G_1 \tmax} \norma{u_o}
      _{\L1(\reali^+; \reali^n)}
      + \tmax\, ^2 2n K_\infty (A_L K_1 + C_\infty)\right]
      \; \tmax \, d_{\XX^n}(w, \bar w) \;
      e^{3 (G_1 + A_L K_1 + C_\infty) \tmax}
    \\
    & + 2 \, n \, A_L \, K_1
      \; \tmax \, d_{\XX^n} (w, \bar w) \; e^{4 (G_1 + A_L K_1 + C_\infty) \tmax}
    \\
    & + 2 \, n \, C_\infty
      \, \tmax \, d_{\XX^n} (w, \bar w) \; e^{4 (G_1 + A_L K_1 + C_\infty) \tmax}
    \\
    & + A_L
      \left[\norma{u^o_1}_{\L1(\reali^+; \reali)}
      + 2 n K_1 (A_L K_1 + C_\infty) \tmax + \norma{\bar b_i}
      _{\L1(J; \reali)}\right] \;
      \tmax\, d_{\XX^n} (w, \bar w) \;
      e^{(2G_1 + A_L K_1 + C_\infty) \tmax}.
  \end{align*}
  Choosing $\tmax$ sufficiently small, we obtain
  $\norma{(\mathcal{T} w)_i - (\mathcal{T}\bar
    w)_i}_{\C0(J;\L1(\reali^+; \reali))} \le \frac{1}{2n} \; d_{\XX^n}
  (w, \bar w)$.  Together with~(\ref{eq:contraction-1}), this implies
  that
  $d_{\XX^n} (\mathcal{T} w, \mathcal{T} \bar w ) \le \frac{1}{2} \,
  d_{\XX^n} (w, \bar w)$, hence $\mathcal T$ is a contraction.

  \paragraph{Existence, Uniqueness and Lipschitz Continuity in Time.}
  On the basis of Definition~\ref{def:scalar}, a map $u$ is a solution
  to~(\ref{eq:2}) on $[0, \tmax]$ if and only if it is a fixed point
  of ${\cal T}$, whence we have existence and uniqueness of the
  solution on the time interval $[0, \tmax]$. The Lipschitz continuity
  of $u_*$ in time follows from~\ref{it:S:5}.

  \paragraph{Dependence on the Boundary and Initial Data.}
  Call $u'$, respectively $u''$, the solution corresponding to the
  boundary datum $\beta'$, respectively $\beta''$, and to the initial
  datum $u_o'$, respectively $u_o''$. In the following, the constants
  $K_1$ and $K_\infty$ satisfy~(\ref{eq:constant-C}) for both
  $u'_o$ and $u''_o$.
  Fix $i \in \left\{1, \cdots, n\right\}$ and $t \in [0, t_*]$.
  Estimate~(\ref{eq:stability-linear-2}) implies that
  \begin{equation}
    \label{eq:estimate-stability-L1-1}
    \begin{split}
      & \norma{u'_i(t) - u''_i(t)}_{\L1\left(\reali^+; \reali\right)}
      \le e^{Mt} \norma{u'_{o,i}(t) - u''_{o,i}(t)}
      _{\L1\left(\reali^+; \reali\right)}
      \\
      & \quad + 2 e^{2\left(G_1 + M\right) t} \int_0^t \int_{\reali^+}
      \sum_{j=1, j\ne i}^n \modulo{\left(\alpha_i [u'(s)](x)\right)_j
        u'_j(s, x) -
        \left(\alpha_i[u''(s)](x)\right)_j u''_j(s,x)} \d s \d x
      \\
      & \quad + 2 e^{2\left(G_1 + M\right) t} \int_0^t \int_{\reali^+}
      \sum_{j=1, j\ne i}^n \modulo{\left(\gamma_i(s,x)\right)_j}
      \modulo{ u'_j(s, x) - u''_j(s,x)} \d s \d x
      \\
      & \quad + e^{2\left(G_1 + M\right) t} \int_0^t
      \modulo{\beta'_i\left(s, \cdots, u'_{i-1}\left(s, \bar
            x_{i-1}-\right) \right) - \beta''_i\left(s, \cdots,
          u''_{i-1}\left(s, \bar x_{i-1}- \right)\right)} \d s
      \\
      & \quad + e^{2\left(G_1 + M\right) t} \left[ \norma{u''_{o,
            i}}_{\L\infty\left(\reali^+; \reali\right)} + 2t F_\infty
        + \frac{1}{\check g}
        \norma{\beta''_i\left(\cdot, \cdots, u''_{i-1} \left(\cdot,
              \bar x_{i-1}-\right) \right)}_{\L\infty\left([0, t];
            \reali\right)} \right] \times
      \\
      & \qquad \times \int_0^t \norma{\left(\alpha_i[u'(s)]\right)_i -
        \left(\alpha_i[u''(s)]\right)_i }_{\L1\left(\reali^+; \reali\right)} \d
      s\,.
    \end{split}
  \end{equation}
  We need to estimate every term in the right hand side
  of~(\ref{eq:estimate-stability-L1-1}). Preliminary,
  using~\ref{hyp:B_i}, (\ref{eq:barB-1}) and~(\ref{eq:vertical-stability-2}),
  we deduce, for $t \in [0, t_*]$, that
  \begin{equation}
    \label{eq:stability-est-beta}
    \begin{split}
      & \int_0^t \modulo{\beta'_i\left(s, \cdots, u'_{i-1}\left(s,
            \bar x_{i-1}-\right) \right) - \beta''_i\left(s, \cdots,
          u''_{i-1}\left(s, \bar x_{i-1}- \right)\right)} \d s
      \\
      \le & B_L \sum_{j=1}^{i-1} \norma{u'_j\left(\cdot, \bar x_j
          -\right) - u''_j\left(\cdot, \bar x_j
          -\right)}_{\L1\left((0,t); \reali\right)}
      +
      \norma{\beta'_i - \beta''_i}_{\L\infty ([0,t] \times [0, K_\infty]^{i-1}; \reali)} \, t
      \\
      \le & B_L\frac{e^{\left(G_1 + M\right) t}}{G_\infty}
      \sum_{j=1}^{i-1} \left[e^{G_1 t}
        \norma{u'_{o,j}}_{\L\infty\left(\reali^+; \reali\right)} + t
        F_\infty\right] \times
      \\
      & \quad \times \int_0^t \int_0^{+\infty}
      \modulo{\left(\alpha_j[u'(s)](x)\right)_j -
        \left(\alpha_j[u''(s)](x)\right)_j} \d x \d s
      \\
      & + B_L e^{Mt} \norma{u'_o - u''_o}_{\L1\left(\reali^+;
          \reali^n\right)} + B_L e^{\left(2 G_1 + M\right) t} \times
      \\
      & \quad \times \sum_{j=1}^i \sum_{h=1, h \ne j}^n \int_0^t
      \int_0^{+\infty} \modulo{\left(\alpha_j[u'(s)](x)\right)_h
        u'_h(s,x) - \left(\alpha_j[u''(s)](x)\right)_h u''_h(s,x)} \d
      x \d s
      \\
      & + B_L e^{\left(2 G_1 + M\right) t} \sum_{j=1}^i \sum_{h=1, h
        \ne j}^n \int_0^t \int_0^{+\infty} \modulo{\left(\gamma_j(s,
          x)\right)_h} \modulo{u'_h(s,x) - u''_h(s,x)} \d x \d s
      \\
      & +
      \norma{\beta'_i - \beta''_i}_{\L\infty ([0,t] \times [0, K_\infty]^{i-1}; \reali)} \, t.
    \end{split}
  \end{equation}
  For $j, h \in \left\{1, \cdots, n\right\}$, $j \ne h$, and $t \in [0, t_*]$,
  using~\ref{hyp:A_i}, \eqref{eq:hyp-A1}, \eqref{eq:hyp-A2},
  and~(\ref{eq:constant-C}), we have
  \begin{equation}
    \label{eq:estimate:alpha-u}
    \begin{split}
      & \int_0^t \int_{\reali^+} \modulo{\left(\alpha_j
          [u'(s)](x)\right)_h u'_h(s, x) -
        \left(\alpha_j[u''(s)](x)\right)_h u''_h(s,x)} \d x \d s
      \\
      \le
      & \int_0^t \int_{\reali^+} \modulo{\left(\alpha_j
          [u'(s)](x)\right)_h \left(u'_h(s, x) - u''_h(s,x)\right)} \d
      x \d s
      \\
      & + \int_0^t \int_{\reali^+} \modulo{\left(\alpha_j [u'(s) -
          u''(s)](x)\right)_h u''_h(s, x)} \d x \d s
      \\
      \le
      & A_L K_1 \int_0^t \norma{u'_h(s) - u''_h(s)}
      _{\L1\left(\reali^+; \reali\right)} \d s + A_1 K_\infty \int_0^t
      \norma{u'(s) - u''(s)} _{\L1\left(\reali^+; \reali^n\right)} \d
      s.
    \end{split}
  \end{equation}
  Moreover, for $j, h \in \left\{1, \cdots, n\right\}$, $j \ne h$,
  and $t \in [0, t_*]$,
  using~\ref{hyp:gamma_i}, we get
  \begin{equation}
    \label{eq:estimate:gamma}
    \int_0^t \int_{\reali^+}
    \modulo{\left(\gamma_j(s,x)\right)_h}
    \modulo{ u'_h(s, x) - u''_h(s,x)} \d x \d s
    \le C_\infty \int_0^t
    \norma{u'_h(s) - u''_h(s)}_{\L1\left(\reali^+; \reali\right)} \d s.
  \end{equation}
  Finally, for $j \in \left\{1, \cdots, n\right\}$ and $t \in [0, t_*]$,
  using~\ref{hyp:A_i} and~\eqref{eq:hyp-A2},
  we have that
  \begin{equation}
    \label{eq:estimate:alpha-i2}
    \begin{split}
      \int_0^t \norma{\left(\alpha_j[u'(s)]\right)_j -
        \left(\alpha_j[u''(s)]\right)_j }_{\L1\left(\reali^+; \reali\right)} \d
      s
      & = \int_0^t \int_0^{+\infty} \modulo{\left(\alpha_j[u'(s) - u''(s)]
        \right)_j (x)} \d x \d s
      \\
      & \le A_1 \int_0^t \norma{u'(s) - u''(s)}_{\L1\left(\reali^+; \reali^n\right)}
      \d s.
    \end{split}
  \end{equation}
  Inserting~(\ref{eq:stability-est-beta}), (\ref{eq:estimate:alpha-u}),
  (\ref{eq:estimate:gamma}), and~(\ref{eq:estimate:alpha-i2})
  into~(\ref{eq:estimate-stability-L1-1}) we obtain that, for $t \in [0, t_*]$,
  \begin{align*}
    \norma{u'(t) - u''(t)}_{\L1\left(\reali^+; \reali^n\right)}
    & \le \mathcal H_1(t) \int_0^t \norma{u'(s) - u''(s)}
      _{\L1\left(\reali^+; \reali^n\right)}
      \d s
    \\
    & \quad + \mathcal H_2(t) \norma{u'_{o}(t) - u''_{o}(t)}
      _{\L1\left(\reali^+; \reali^n\right)}
    \\
    & \quad +
      e^{2 (G_1+M)t}
      \norma{\beta' - \beta''}_{\L\infty ([0,t] \times [0, K_\infty]^n; \reali^n)}
      \, t,
  \end{align*}
    where
    \begin{align*}
      \mathcal H_1(t)
      & = n e^{2 \left(G_1 + M\right) t} \left[2 A_L K_1 + 2n A_1 K_\infty
        + 2 C_\infty \right]
      \\
      & \quad + n^2 e^{3 \left(G_1 + M\right) t} \frac{B_L A_1}{G_\infty}
        \left[e^{G_1 t} \norma{u'_o}_{\L\infty\left(\reali^+; \reali^n\right)}
        + t F_\infty\right]
      \\
      & \quad + n^2 e^{\left(4 G_1 + 3 M\right) t} B_L
        \left[A_L K_1 + n A_1 K_\infty + C_\infty\right]
      \\
      & \quad + n^2 e^{2 \left(G_1 + M\right) t} A_1
        \left[\norma{u''_o}_{\L\infty\left(\reali^+; \reali^n\right)}
        + 2 n t F_\infty + \frac{n}{\check g} \left(B_\infty + n B_L K_\infty
        \right)\right] \,;
      \\
      \mathcal H_2(t)
      & = n e^{Mt} \left(1 + B_L e^{2 \left(G_1 + M\right) t}\right)\,.
    \end{align*}
    An application of Gronwall Lemma yields~\eqref{eq:41}.
  \end{proofof}

\begin{proofof}{Corollary~\ref{cor:posGlog}}
  We proceed with the same notation used in the proof of
  Theorem~\ref{thm:general}, $u$ being the solution to~\eqref{eq:2} on
  $J$.

  The positivity of each $u_i$ directly follows
  from~(\ref{eq:12})--(\ref{eq:11}).

  Assume, by contradiction, that there exists a maximal time of
  existence $\bar t$ for the solution $u$ to~\eqref{eq:2}.  A direct
  consequence of~\textbf{(NEG)} and~\textbf{(EQ)} is that
  \begin{displaymath}
    \partial_t \left(\sum_{i=1}^n u_i\right)
    +
    \partial_x \left(
      g (t,x) \sum_{i=1}^n u_i
    \right)
    =
    \left(\sum_{i=1}^n \left(
        \alpha_i[u (t)] + \gamma_i (t,x)
      \right)\right) \cdot u
    \leq 0 \,.
  \end{displaymath}
  Therefore, \ref{it:S:2} and~\ref{it:S:3} ensure that
  $\norma{\sum_{i=1}^n u_i (t)}_{\L1 (\reali^+; \reali)}$,
  $\norma{\sum_{i=1}^n u_i (t)}_{\L\infty (\reali^+; \reali)}$ and
  $\tv\left(\sum_{i=1}^n u_i (t); \reali^+\right)$ are uniformly
  bounded on $[0, \bar t[$. The uniform continuity of $u$ in time
  ensures that $u$ can be defined also at time $\bar t$. A further
  application of Theorem~\ref{thm:general} allows to uniquely extend $u$
  beyond time $\bar t$, yielding the contradiction.

  The functional Lipschitz continuous dependence of $u_*$ on the
  initial datum $u_o$ and on the boundary inflow $\beta$ now follows
  from~\ref{it:SP8} and~\ref{it:SP10}, possibly iterating the
  estimate~\eqref{eq:vertical-stability} to comply with the condition
  $\gamma (t) < \bar x$ therein.
\end{proofof}

\subsection{Proofs Related to Section~\ref{sec:CSIR} -- About the Models~\eqref{eq:13}--\eqref{eq:1} and~\eqref{eq:13}--\eqref{eq:14}}
\label{sec:ProofsCSIR}

\begin{proofof}{Theorem~\ref{thm:wp}}
  The present proof consists in showing that with the present
  assumptions, Theorem~\ref{thm:general} and
  Corollary~\ref{cor:posGlog} can be applied
  to~\eqref{eq:13}--\eqref{eq:1}--\eqref{eq:IC+BDY}. To this aim, set
  for notational simplicity $\bar a_0 = 0$, $\bar a_{N+1} = +\infty$
  and define
  \begin{equation}
    \label{eq:31}
    \begin{array}{rcl}
      u_{1+3j} (t,a)
      & =
      & S (t,a+\bar a_j)
      \\
      u_{2+3j} (t,a)
      & =
      & I (t,a+\bar a_j)
      \\
      u_{3+3j} (t,a)
      & =
      & R (t,a+\bar a_j)
    \end{array}
    \qquad\qquad\qquad \mbox{ for }
    \begin{array}{r@{\,}c@{\,}l}
      j
      & =
      & 0,1, \ldots, N
      \\
      t
      & \in
      & \II
      \\
      a
      & \in
      & \reali^+ \,.
    \end{array}
  \end{equation}
  Set $g_i (t,a) = 1$ for $i=1, \ldots, n$ and $n = 3N+3$.  Define for
  $j=0,1, \ldots, N$ and $i = 1, \ldots, n$
  \begin{equation}
    \label{eq:10}
    \begin{array}{@{}rcl@{}}
      \left[\alpha_{1+3j}[u] (a)\right]_i
      & =
      & \left\{
        \begin{array}{l@{\qquad\quad}r}
          -
          \displaystyle \sum_{\ell=1}^{N + 1}
          \int_{\bar a_{\ell-1}}^{\bar a_\ell}
          \lambda (a + \bar a_j, a') \;
          u_{2+3\ell} (a' - \bar a_{\ell - 1}) \d{a'}
          & i = 1+3j
          \\
          0
          & \mbox{otherwise}
        \end{array}
            \right.
      \\
      \left[\alpha_{2+3j}[u] (a)\right]_i
      & =
      & \left\{
        \begin{array}{l@{\qquad\quad}r}
          \displaystyle
          \sum_{\ell=1}^{N + 1}
          \int_{\bar a_{\ell-1}}^{\bar a_\ell}
          \lambda (a + \bar a_j, a') \;
          u_{2+3\ell} (a' - \bar a_{\ell - 1}) \d{a'}
          & i = 1+3j
          \\
          0
          & \mbox{otherwise}
        \end{array}
            \right.
      \\
      \left[\alpha_{3+3j}[u] (a)\right]_i
      & =
      & 0
    \end{array}
  \end{equation}
  \begin{equation}
    \label{eq:28}
    \begin{array}{rcl}
      \left[\gamma_{1+3j} (t,a)\right]_i
      & =
      & \left\{
        \begin{array}{l@{\qquad\qquad}r}
          -d_S (t,a+\bar a_j)
          & i = 1+3j
          \\
          0
          & \mbox{otherwise}
        \end{array}
            \right.
      \\
      \left[\gamma_{2+3j} (t,a)\right]_i
      & =
      & \left\{
        \begin{array}{l@{\qquad\qquad}r}
          - d_I (t,a+\bar a_j)
          - r_I (t,a+\bar a_j)
          & i = 2+3j
          \\
          0
          & \mbox{otherwise}
        \end{array}
            \right.
      \\
      \left[\gamma_{3+3j} (t,a)\right]_i
      & =
      & \left\{
        \begin{array}{l@{\qquad\qquad}r}
          - d_R (t,a+\bar a_j)
          & i = 3+3j
          \\
          r_I (t,a+\bar a_j)
          & i = 2+3j
          \\
          0
          & \mbox{otherwise}
        \end{array}
            \right.
    \end{array}
  \end{equation}
  Concerning the boundary conditions and transmission relations, we
  set $\bar x_i = \bar a_j - \bar a_{j-1}$ for all
  $i \in \left\{1, \cdots, n\right\}$ and for all
  $j=\left\{1, \cdots, N + 1\right\}$ such that
  $i - 3 (j - 1) \in \{1, 2, 3\}$.  Moreover
  \begin{equation}
    \label{eq:29}
    \begin{array}{@{}r@{\,}c@{\,}l@{\qquad}r@{\,}c@{\,}l@{\qquad}r@{\,}c@{\,}l@{}}
      \beta_1 (t,u)
      & =
      & S_b (t)
      & \beta_2 (t,u)
      & =
      & I_b (t)
      & \beta_3 (t,u)
      & =
      & R_b (t)
      \\
      \beta_{3j+1} (t,u)
      & =
      & \left(1-\eta_j (t)\right) u_{3j-2}
      & \beta_{3j+2} (t,u)
      & =
      & u_{3j-1}
      & \beta_{3j+3} (t,u)
      & =
      & \eta_j (t) \, u_{3j-2} + u_{3j}
    \end{array}
  \end{equation}

  We now verify that the assumptions required in
  Theorem~\ref{thm:general} on the functions above hold.

  \subparagraph{\ref{hyp:g} holds.} It is immediate, since
  $g_i (t,a)=1$ for all $(t,a) \in \II \times \reali^+$.

  \subparagraph{\ref{hyp:A_i} holds.} On the basis of~\eqref{eq:10},
  we have:
  \begin{flalign*}
    \norma{\alpha_i[w]}_{\L\infty(\reali^+; \reali^n)} & \leq
    \Lambda_\infty \, \norma{u}_{\L1 (\reali^+; \reali^n)} &
    \mbox{[By~\eqref{eq:Lambda2}, \eqref{eq:hyp-A1} holds with
      $A_L = \Lambda_\infty$]}
    \\
    \tv (\alpha_i[w]; (\reali^+)^n) & \le \sum_{\ell=1}^{N+1}
    \int_{\bar a_{\ell-1}}^{\bar a_\ell} \tv (\lambda (\cdot, a');
    \reali^+) \, \modulo{w_{2+3\ell} (a'-\bar a_{\ell-1})} \, \d{a'}
    \hspace{-5cm}& \mbox{[By~\ref{eq:TV3}]}
    \\
    & \leq \Lambda_\infty \norma{u}_{\L1 (\reali^+; \reali^n)} &
    \mbox{[By~\eqref{eq:Lambda2}, \eqref{eq:hyp-A1} holds with
      $A_L = \Lambda_\infty$]}
  \end{flalign*}
  \begin{flalign*}
    \norma{\alpha_i[u](a_1) {-} \alpha_i[u](a_2)} & \le
    \sum_{\ell=1}^{N+1} \! \int_{\bar a_{\ell-1}}^{\bar a_\ell} \!
    \modulo{\lambda (a_1 {+} \bar a_j, a') {-} \lambda (a_2 {+} \bar a_j,
      a')} \, \modulo{u_{2+3\ell}} (a'-\bar a_{\ell-1}) \d{a'} &
    \mbox{[By~\ref{hyp:lambda}]}
    \\
    & \le \Lambda_l \, \sum_{\ell=1}^{N+1} \int_{\bar
      a_{\ell-1}}^{\bar a_\ell} \modulo{u_{2+3\ell}} (a'-\bar
    a_{\ell-1}) \, \d{a'} \; \modulo{a_1-a_2} &
    \mbox{[By~\eqref{eq:Lambda}]}
    \\
    & \leq N \, \Lambda_L \, \norma{u}_{\L1(\reali^+; \reali^n)} \;
    \modulo{a_1-a_2}
  \end{flalign*}
  proving~\eqref{eq:hyp-A4} in~\ref{hyp:A_i} with
  $A_2 = N \, \Lambda_L \, \norma{u}_{\L1(\reali^+; \reali^n)}$.

  \subparagraph{\ref{hyp:gamma_i} holds.} Refer to~\eqref{eq:28}. The
  Lipschitz continuity of $\gamma$ directly follows
  from~\eqref{eq:22}, proving~\eqref{eq:15}. The other conditions are
  immediate consequences of~\eqref{eq:24}, \eqref{eq:27}
  and~\eqref{eq:24}.

  \subparagraph{\ref{hyp:B_i} holds.} Refer
  to~\eqref{eq:29}. Condition~\eqref{eq:barB-1} is immediate, thanks
  to the boundedness of $\eta$. \eqref{eq:barB-2}, \eqref{eq:barB-3}
  and~\eqref{eq:barB-4} follow from~\eqref{eq:iBV-data}.

  \subparagraph{\ref{hyp:init-boundary} holds.} It is an immediate
  consequence of~\eqref{eq:iBV-data}, using~\eqref{eq:31} at $t=0$.

  \subparagraph{\ref{it:Pos} holds.} It immediately follows
  from~\ref{hyp:data}.

  \subparagraph{\ref{it:Neg} holds.} Long but straightforward
  computations based on~\eqref{eq:13}, \eqref{eq:10} and~\eqref{eq:28}
  show that~\ref{it:Neg} holds.

  \subparagraph{\ref{it:Eq} holds.} It is a direct consequence
  of~\eqref{eq:13}.

  \subparagraph{Dependence on $\eta$.} Theorem~\ref{thm:general} and
  Corollary~\ref{cor:posGlog} ensure the Lipschitz continuous
  depen\-dence of the solution $(S,I,R)$ in $\L1$ on the boundary datum
  through its $\L1$ norm. In view of~\eqref{eq:29}, this yields the
  continuous dependence of the solution $(S,I,R)$ in $\L1$ on $\eta$
  in $\L\infty$.
\end{proofof}

The proof of Theorem~\ref{thm:wp2} can be obtained from the one above
exchanging the roles of the independent variables $t$ and
$a$. However, for completeness, we provide an independent proof.

\begin{proofof}{Theorem~\ref{thm:wp2}}
  We now show that Theorem~\ref{thm:general} can be iteratively
  applied to
  problem~\eqref{eq:13}--\eqref{eq:IC+BDY}--\eqref{eq:14}. To this
  aim, define $n=3$ and
  \begin{displaymath}
    u_1 (t,a)
    =
    S (t,a)
    \,,\quad
    u_2 (t,a)
    =
    I (t,a)
    \,,\quad
    u_3 (t,a)
    =
    R (t,a) \,.
  \end{displaymath}
  \begin{displaymath}
    \begin{array}{@{}r@{\,}c@{\,}l@{\quad}r@{\,}c@{\,}l@{}}
      \left[\alpha_1[u] (a)\right]_i
      & =
      & \left\{
        \begin{array}{@{}ll}
          - \int_{\reali^+} \lambda (a,a') \, u_2 (a') \d{a'}
          & i=1
          \\
          0
          & i=2,3
        \end{array}
            \right.
      &\left[\gamma_1(t,a)\right]_i
          & =
          & \left\{
            \begin{array}{@{}ll}
              -d_S (t,a)
              & i=1
              \\
              0
              & i=2,3
            \end{array}
                \right.
      \\
      \left[\alpha_2[u] (a)\right]_i
      & =
      & \left\{
        \begin{array}{@{}ll}
          \int_{\reali^+} \lambda (a,a') \, u_2 (a') \d{a'}
          & i=1
          \\
          0
          & i=2,3
        \end{array}
            \right.
      &
        \left[\gamma_2 (t,a)\right]_i
          & =
          & \left\{
            \begin{array}{@{}ll}
              - d_I (t,a) - r_I (t,a)
              & i=2
              \\
              0
              & i=1,3
            \end{array}
                \right.
      \\
      \left[\alpha_3 (a)\right]_i
      & =
      & 0
          & \left[\gamma_3 (t,a)\right]_i
          & =
      & \left\{
        \begin{array}{@{}ll}
          0
          & i=1
          \\
          r_I (t,a)
          & i=2
          \\
          -d_R (t,a)
          & i=3
        \end{array}
            \right.
    \end{array}
  \end{displaymath}
  We now iteratively apply Theorem~\ref{thm:general} on the time
  interval $[\bar t_k, \bar t_{k+1}]$ assigning, on the basis
  of~\eqref{eq:IC+BDY}, the initial and boundary data
  \begin{equation}
    \label{eq:30}
    \begin{array}{lll}
      \begin{array}{r@{\,}c@{\,}l}
        k
        & =
        & 0
        \\
        t
        & \in
        & [0, \bar t_1]
        \\
        a
        & \in
        & \reali^+
      \end{array}
        & \left\{
          \begin{array}{rcl@{}}
            u_1^o (a)
            & =
            & S_o (a)
            \\
            u_2^o (a)
            & =
            & I_o (a)
            \\
            u_3^o (a)
            & =
            & R_o (a)
          \end{array}
              \right.
        & \begin{array}{rcl@{}}
            \beta_1 (t,u)
            & =
            & S_b (t)
            \\
            \beta_2 (t,u)
            & =
            & I_b (t)
            \\
            \beta_3 (t,u)
            & =
            & R_b (t)
          \end{array}
      \\
      \begin{array}{r@{\,}c@{\,}l}
        k
        & \geq
        & 1
        \\
        t
        & \in
        &  [\bar t_k, \bar t_{k+1}]
        \\
        a
        & \in
        & \reali^+
      \end{array}
        & \left\{
          \begin{array}{rcl@{}}
            u_1^o (\bar t_k, a)
            & =
            & \left(1-\nu_k (a)\right) \, u_1 (\bar t_k-,a)
            \\
            u_2^o (\bar t_k, a)
            & =
            & u_2 (\bar t_k-,a)
            \\
            u_3^o (\bar t_k, a)
            & =
            & u_3
              (\bar t_k-,a) + \nu_k (a) \, u_1 (\bar t_k-,a)
          \end{array}
              \right.
        & \begin{array}{rcl@{}}
            \beta_1 (t,u)
            & =
            & S_b (t)
            \\
            \beta_2 (t,u)
            & =
            & I_b (t)
            \\
            \beta_3 (t,u)
            & =
            & R_b (t)
          \end{array}
    \end{array}
  \end{equation}

  We now verify that the assumptions required in
  Theorem~\ref{thm:general} on the functions above hold.

  \subparagraph{\ref{hyp:g} holds.} It is immediate, since
  $g_i (t,a)=1$ for all $(t,a) \in \II \times \reali^+$.

  \subparagraph{\ref{hyp:A_i} holds.}
  \begin{flalign*}
    \norma{\alpha_i[w]}_{\L\infty(\reali^+; \reali^n)} & \leq
    \Lambda_\infty \, \norma{u}_{\L1 (\reali^+; \reali^n)} &
    \mbox{[By~\eqref{eq:Lambda2}, \eqref{eq:hyp-A1} holds with
      $A_L = \Lambda_\infty$]}
    \\
    \tv (\alpha_i[w]; (\reali^+)^n) & \le \int_{\reali^+} \!\!\! \tv
    (\lambda (\cdot, a'); \reali^+) \modulo{w_{j} (a')} \d{a'}
    \!\!\!\!\!\!\!\!\!  & \mbox{[By~\ref{eq:TV3}]}
    \\
    & \leq \Lambda_\infty \norma{u}_{\L1 (\reali^+; \reali^n)} &
    \mbox{[By~\eqref{eq:Lambda2}, \eqref{eq:hyp-A1} holds with
      $A_L = \Lambda_\infty$]}
    \\
    \norma{\alpha_i[u](a_1) {-} \alpha_i[u](a_2)} & \le
    \int_{\reali^+} \!\!\! \modulo{\lambda (a_1, a') - \lambda (a_2,
      a')} \modulo{u_{2}} (a') \d{a'} \hspace{-80mm}
    \\
    & \le \Lambda_l \, \int_{\reali^+} \modulo{u_{2} (a')} \, \d{a'}
    \; \modulo{a_1-a_2} & \mbox{[By~\eqref{eq:Lambda}]}
    \\
    & \leq \Lambda_L \, \norma{u}_{\L1(\reali^+; \reali^n)} \;
    \modulo{a_1-a_2}\,. & \mbox{[\eqref{eq:hyp-A4} holds with
      $A_2 = \Lambda_L\norma{u}_{\L1(\reali^+; \reali^n)}$]}
  \end{flalign*}

  \subparagraph{\ref{hyp:gamma_i} holds.} The Lipschitz continuity of
  $\gamma$ directly follows from~\eqref{eq:22},
  proving~\eqref{eq:15}. The other conditions are immediate
  consequences of~\eqref{eq:24}, \eqref{eq:27} and~\eqref{eq:24}.

  \subparagraph{\ref{hyp:B_i} holds.} The definitions~\eqref{eq:30}
  and~\eqref{eq:iBV-data} directly imply~\eqref{eq:barB-2},
  \eqref{eq:barB-3} and~\eqref{eq:barB-4}.

  \subparagraph{\ref{hyp:init-boundary} holds.} It directly follows
  from~\eqref{eq:30}.

  \subparagraph{\ref{it:Pos} holds.} It immediately follows
  from~\ref{hyp:data}.

  \subparagraph{\ref{it:Neg} holds.} Long but straightforward
  computations based on~\eqref{eq:13}, \eqref{eq:10} and~\eqref{eq:28}
  show that~\ref{it:Neg} holds.

  \subparagraph{\ref{it:Eq} holds.} It is a direct consequence
  of~\eqref{eq:13}.

  \subparagraph{Dependence on $\nu$.} Repeat the same argument used in
  the final part of the proof of Theorem~\ref{thm:wp}, replacing the
  boundary datum with the initial datum.
\end{proofof}

\bigskip

\noindent\textbf{Acknowledgments:} Part of this work was supported by the  PRIN~2015 project \emph{Hyperbolic Systems of Conservation Laws
  and Fluid Dynamics: Analysis and Applications} and by the
GNAMPA~2018 project \emph{Conservation Laws: Hyperbolic Games,
  Vehicular Traffic and Fluid dynamics}. The \emph{IBM Power Systems
  Academic Initiative} substantially contributed to the numerical
integrations.

{\small

  \bibliographystyle{abbrv}

  \bibliography{sir_6.bib} }

\end{document}